\documentclass[11pt,reqno]{amsart}

\usepackage[legalpaper, portrait, margin=1.0in]{geometry}

\usepackage{amsmath,amsthm,amssymb,amsfonts,amstext,mathtools,thmtools}
\usepackage{enumerate}
\usepackage{url,hyperref}
\usepackage{graphicx}
\usepackage{xcolor}
\usepackage{verbatim}
\usepackage{comment}
\usepackage{caption,subcaption}
\usepackage{pgfplots}
    \usetikzlibrary{patterns,trees,arrows,intersections}
    \usepgfplotslibrary{fillbetween}
    \pgfdeclarelayer{pre main}
    \pgfplotsset{compat=newest}
\usepackage{etoolbox}
    \makeatletter
    \patchcmd{\@setaddresses}{\indent}{\noindent}{}{}
    \patchcmd{\@setaddresses}{\indent}{\noindent}{}{}
    \patchcmd{\@setaddresses}{\indent}{\noindent}{}{}
    \patchcmd{\@setaddresses}{\indent}{\noindent}{}{}
    \makeatother

\allowdisplaybreaks

\DeclareMathOperator{\Id}{Id}
\DeclareMathOperator{\slp}{slope}

\newcommand{\R}{\mathbb{R}}

\newcommand{\Z}{\mathbb{Z}}

\newcommand{\mmid}{\,\middle\vert \,}

\newcommand{\slr}{{\rm SL}_2(\R)}
\newcommand{\slz}{{\rm SL}_2(\Z)}

\newcommand{\slxo}{{\rm SL}(X,\omega)}
\renewcommand{\Re}{{\rm Re}}
\renewcommand{\Im}{{\rm Im}}

\newcommand{\bbC}{\mathbb{C}}
\newcommand{\bbR}{\mathbb{R}}
\newcommand{\bbN}{\mathbb{N}}
\newcommand{\bbZ}{\mathbb{Z}}

\newcommand{\bbS}{\mathbb{S}}
\newcommand{\bbG}{\mathbb{G}}

\newcommand{\ath}{\mathrm{arctanh}}

\theoremstyle{remark}

\newtheorem{thm}{Theorem}
\numberwithin{thm}{section}

\newtheorem{lem}[thm]{Lemma}

\theoremstyle{definition}
\newtheorem{defn}[thm]{Definition}

\newtheorem{innercustomthm}{Theorem}
\newenvironment{mythm}[1]
  {\renewcommand\theinnercustomthm{#1}\innercustomthm}
  {\endinnercustomthm}

\theoremstyle{remark}

\begin{document}

\title[Slope Gap Distribution of the Double Heptagon]{Slope Gap Distribution of the Double Heptagon and an Algorithm for Determining Winning Vectors}

\author[F.~Al~Assal]{Fernando Al Assal}
\address{Fernando Al Assal\\
Department of Mathematics\\
University of Wisconsin, Madison\\
480 Lincoln Dr.\\
Madison, WI 53706}
\email{\href{alassal@wisc.edu}{alassal@wisc.edu}}

\author[N.~Ali]{Nada Ali}
\address{Nada Ali\\
Department of Mathematics\\
Rice University\\
6100 Main St.\\
Houston, TX 77005}
\email{\href{nada.ali@rice.edu}{nada.ali@rice.edu}}

\author[U.~Arengo]{Uma Arengo}
\address{Uma Arengo\\
Department of Mathematics \\
Yale University\\
219 Prospect St.\\
New Haven, CT 06520}
\email{\href{mailto:uma.arengo@yale.edu}{uma.arengo@yale.edu}}

\author[T.~McAdam]{Taylor McAdam}
\address{Taylor McAdam\\
Department of Mathematics \& Statistics 
\\ Pomona College
\\ 333 N. College Way
\\ Claremont, CA 91711}
\email{\href{mailto:taylor.mcadam@pomona.edu}{taylor.mcadam@pomona.edu}}

\author[C.~Newman]{Carson Newman}
\address{Carson Newman\\
Department of Mathematics\\
University of California, Berkeley\\
970 Evans Hall\\
Berkeley, CA 94720}
\email{\href{mailto:carsonnewman@berkeley.edu}{carsonnewman@berkeley.edu}}

\author[N.~Scully]{Noam Scully}
\address{Noam Scully\\
Department of Mathematics \\
Yale University\\
219 Prospect St.\\
New Haven, CT 06520}
\email{\href{mailto:noam.scully@yale.edu}{noam.scully@yale.edu}}

\author[S.~Zhou]{Sophia Zhou}
\address{Sophia Zhou\\
Department of Mathematics \\
Yale University\\
219 Prospect St.\\
New Haven, CT 06520}
\email{\href{mailto:sophia.zhou@yale.edu}{sophia.zhou@yale.edu}}

\begin{abstract} 
In this paper, we study the distribution of renormalized gaps between slopes of saddle connections on translation surfaces. Specifically, we describe a procedure for finding the ``winning holonomy vectors'' as defined in \cite{KSW}, which constitutes a key step in calculating the slope gap distribution for an arbitrary Veech surface. We then apply this method to explicitly compute the gap distribution for the regular double heptagon translation surface. This extends work of \cite{ACL} on the gap distribution for the ``golden L'' translation surface, which is equivalent to the regular double pentagon surface. 
\end{abstract}


\maketitle


\section{Introduction}

A \emph{translation surface} can be defined as a collection of polygons in the plane with pairs of equal length, parallel sides of opposite orientation glued by translations. Two collections of polygons are considered equivalent (i.e.\ representing the same translation surface) if you can get from one to the other by a series of cut-translate-paste moves. Equivalently, a translation surface can be defined as a compact Riemann surface $X$ together with a holomorphic $1$-form $\omega$ defined on $X$, and we will denote a translation surface by the pair $(X,\omega)$. The resulting surfaces are locally flat except at a finite number of \emph{cone points} where all of the curvature is concentrated. A cone point of total angle $2\pi(k+1)$ for $k\in\bbN$ corresponds to a zero of order $k$ for the holomorphic $1$-form, and by the Riemann-Roch theorem, the sum of the degrees of excess angle for all of the cone points equals $2g-2$, where $g$ is the genus of the surface. The number of cone points and their excess angles naturally partition the space of all translation surfaces into \emph{strata}, where $\mathcal{H}(k_1,\dots,k_n)$ denotes the stratum of translation surfaces with $n$ cone points of angles $2\pi(k_i+1)$ for $i=1,\dots,n$. A \emph{saddle connection} on a translation surface is a geodesic straight-line path that connects a cone point to a cone point (not necessarily distinct) without passing through any cone points in between. 
The group $\slr$ acts naturally on the space of translation surfaces (realized as polygons) by its action on the plane. The \emph{Veech group} of a translation surface $(X,\omega)$ is the stabilizer of the translation surface under this $\slr$-action, denoted $\slxo$, and a surface is called a \emph{Veech surface} or \emph{lattice surface} if its Veech group is a lattice in $\slr$, which can be understood as saying that the translation surface has a large amount of symmetry. 
We refer to \cite{HS, Zorich, WrightSurvey} for more background on the general theory of translation surfaces. 

\subsection{Past Work}

The saddle connections of a translation surface capture important information about the surface's geometry, and a significant amount of past work has been dedicated to understanding the distribution of saddle connections on different translation surfaces. In \cite{Masur88, Masur90}, Masur shows that for any fixed translation surface, there are quadratic upper and lower bounds for the number of saddle connections of length at most $R$ as $R$ tends to infinity. In \cite{Veech89}, Veech shows that there is an exact quadratic asymptotic for the number of saddle connections of length at most $R$ for any Veech surface. In \cite{Veech98}, Veech shows that there is an exact quadratic asymptotic for the number of saddle connections when averaged over all translation surfaces in a given stratum, as well as that the directions of saddle connections in a Veech surface equidistribute in $S^1$ with respect to the Lebesgue measure. Eskin-Masur show in \cite{EskinMasur} that almost every surface in a given stratum has an exact quadratic asymptotic for the growth rate of the number of saddle connections of bounded length, and Vorobets shows in \cite{Vorobets} that the directions of saddle connections equidistribute in $S^1$ for almost every translation surface in a stratum. In \cite{Dozier}, Dozier shows that on any translation surface, the uniform measure taken over saddle connections of length at most $R$ converges to the area measure induced by the flat metric as $R\to \infty$. He also shows that the weak limit of any subsequence of the normalized counting measure given by the angles of saddle connections of length at most $R$ is in the Lebesgue measure class on $S^1$ (and if the limit exists as $R\to\infty$, then it converges to Lebesgue measure).

All of these results point to a certain degree of ``randomness'' or uniformity in the distribution of saddle connections on translation surfaces. However, more recent attention has been given to finer statistics in the distribution of saddle connections, suggesting more complicated and subtle behavior. Athreya and Chaika show in \cite{AthCha} that the distribution of renormalized gaps between saddle connection directions exists and is the same for almost every translation surface in a stratum, and moreover that the distribution of renormalized gaps has no support at zero (i.e.\ ``no small gaps'') if and only if the surface is Veech. Inspired by Marklof-Str\"ombergsson \cite{MarkStrom}, their methods reduce the problem of studying gaps between saddle connections to certain equidistribution problems on homoegeneous spaces, setting the stage for much of the work on gaps that follows. 

In \cite{AC13}, Athreya-Cheung use a similar dynamical approach to study the distribution of Farey sequences. The Farey sequence of level $n$, defined as the ordered list of rational numbers between $0$ and $1$ of denominator at most $n$, can be realized as slopes of saddle connections of bounded horizontal length on the square torus translation surface with a single marked point.
Athreya-Cheung rederive the Hall distribution for the limiting behavior of renormalized gaps between Farey fractions (i.e.\ slopes of saddle connections on the square torus) by realizing the BCZ map of \cite{BCZ} as the first-return map for the horocycle flow to a particular Poincar\'e section in $\slr/\slz$. This distribution, first derived by Hall in \cite{Hall}, displays the ``no small gaps'' property consistent with the results of \cite{AthCha} for Veech surfaces, as well as two points of non-differentiability and a quadratic tail. This demonstrates that the distribution of gaps between saddle connection directions may be far from ``random,'' in the sense that it does not resemble the exponential distribution that would result from a sequence of i.i.d. uniform random variables. The remarkable relationship between slopes of saddle connections and questions of number-theoretic interest (such as the distribution of Farey sequences) serves as one motivation for studying the \emph{slopes} (rather than angles) of saddle connections, which is seen throughout the work that follows. The methods of \cite{AC13} are also more naturally suited to studying slopes rather than angles, for reasons that are made clear in Section~\ref{sec:prelims}.

In \cite{ACL}, Athreya-Chaika-Leli\`evre find the gap distribution for slopes of saddle connections on the ``golden L'' translation surface, which is equivalent to the regular double pentagon surface, and part of the inspiration for the present work on the double heptagon. In \cite{UW}, Uyanik-Work calculate the slope gap distribution for the regular octagon, which is generalized in \cite{Previous_SUMRY}, where the authors find the slope gap distribution for the regular $2n$-gon and provide linear upper and lower bounds on the number of points of non-analyticity for this family of surfaces. 
Uyanik-Work also provide a general algorithm in \cite{UW} for constructing the Poincar\'e section and return time function needed to calculate the gap distribution for any Veech surface. 
This algorithm is modified by Kumanduri-Sanchez-Wang in \cite{KSW} to show that the number of points of non-analyticity is finite and that the tail of the distribution displays quadratic decay for any Veech surface. We use the algorithm of \cite{KSW} as a starting point for our work, and a more detailed description of this method can be found in Section~\ref{sec:KSWalg}.

Other relevant past work on gap distributions includes \cite{Ath13}, in which Athreya proves several meta-theorems about gap distributions in the context of homogeneous dynamics, which in particular apply to studying the slopes and angles of saddle connections on translation surfaces. In \cite{Heersink}, Heersink uses techniques from \cite {FS} to lift the Poincar\'e section of Athreya-Cheung to a quotient of $\slr$ by any finite index subgroup of $\slz$ and uses this to find the gap distributions for certain subsets of Farey fractions. In \cite{Taha1, Taha2}, Taha studies the distribution of generalized Farey fractions by finding a Poincar\'e section for the geodesic and horocycle flows on quotients $\slr$ by Hecke triangle groups. In \cite{Sanchez}, Sanchez computes the gap distribution for the family of surfaces known as doubled slit tori, the first example of an explicit gap distribution for a non-Veech surface. In \cite{Work}, Work finds a Poincar\'e section for the horocycle flow on $\mathcal{H}(2)$ and provides bounds on the return time function, which can be understood as a step toward calculating the ``generic'' distribution described by \cite{AthCha} for a surface in this stratum. More recently, Osman, Southerland, and Wang provide the first effective results for slope gap distributions by showing that estimates converge to the slope gap distribution of a Veech surface with a polynomial rate depending on the temperedness of the Veech group, \cite{OSWEffective}.

\subsection{Main Results}

The algorithms of Uyanik-Work and Kumanduri-Sanchez-Wang systematize much of the work of finding the slope gap distribution for an arbitrary Veech surface. 
However, in both algorithms, there is a key step in the process which is left as a black box---namely, determining the ``winning vectors'' at each point on the Poincar\'e section (for more details and a definition, see Section~\ref{sec:KSWalg} or \cite{KSW}). The ``winning vector'' at a given point determines the return time function at that point, and proving which holonomy vectors win on different regions of the Poincar\'e section is often the most difficult step in deriving the slope gap distribution for a surface. In past work this has remained an \emph{ad hoc} part of the process, unique to each surface studied. In Section~\ref{sec:algorithm}, we describe a general method that greatly simplifies the process of determining the winning vectors for an arbitrary Veech surface.  

We demonstrate the utility of this algorithm by using it to compute the gap distribution for the double heptagon in Section~\ref{sec:dbl_heptagon}. 

\begin{mythm}{\ref{thm:double-heptagon-sgd}}
    The renormalized slope gap distribution for the double heptagon is the piecewise real analytic function with $13$ points of non-analyticity shown in Figure~\ref{fig:7gon_slope_gap_dist-1}.
\begin{figure}[h!]
     \centering
     \includegraphics[width=0.6\linewidth]{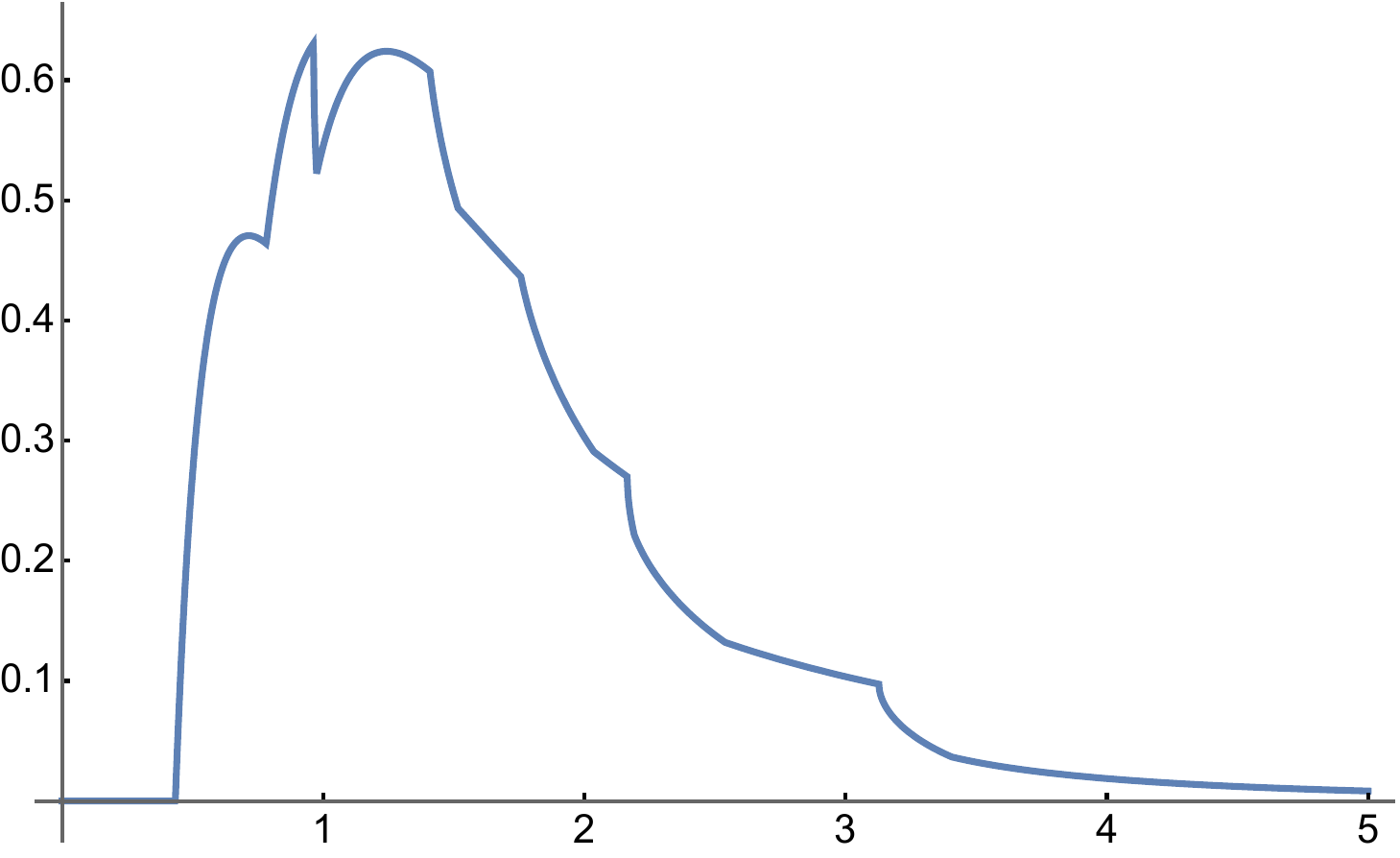}
     \caption{The slope gap distribution of the regular double heptagon translation surface.}

     \label{fig:7gon_slope_gap_dist-1}
\end{figure}
\end{mythm}

This distribution is interesting in its own right, both as an extension of the work of \cite{ACL} for the double pentagon and as an addition to the list of surfaces whose gap distributions have been explicitly computed. However, we believe that our method for determining winning vectors may be of more general interest to the study of slope gap distributions on Veech surfaces.

\subsection{Acknowledgments}

This work originated as part of the 2023 Summer Undergraduate Mathematics Research at Yale (SUMRY) program, and the authors are grateful to Yale University and NSF Award DMS-2050398 for supporting this program which made this work possible. The fourth named author is also grateful for support from the NSF Postdoctoral Fellowship DMS-1903099. 
The authors would like to thank Jordan Grant, Jo O'Harrow, and Hamilton Wan for providing code from past SUMRY projects, as well as Caglar Uyanik for helpful conversations. 

\section{Preliminaries}
\label{sec:prelims}

Let $(X,\omega)$ be a translation surface and let $\slxo$ be its Veech group under the action of $\slr$. If $\gamma$ is an oriented saddle connection on $(X,\omega)$, then the \emph{holonomy vector} associated with $\gamma$ is 
\[
v_\gamma = \int_\gamma \omega \ \in \bbC
\]
which records the length and direction of the saddle connection. Denote the set of all holonomy vectors on $(X,\omega)$ by $\Lambda(X,\omega)$, or simply $\Lambda$ if the surface is understood from context. It is known that $\Lambda$ is a discrete subset of $\bbC$. Consider the set of slopes in $[0,1]$ of holonomy vectors in $\Lambda$ in the first quadrant with real part at most $R$:\footnote{In analogy with the treatment of Farey fractions in \cite{AC13} and because of the connection to the horocycle flow, it makes most sense to bound the horizontal component, rather than the absolute length, when dealing with the distribution of slopes. However, in order to obtain a finite set, we also restrict to slopes in $[0,1]$. Though arbitrary, this choice is of little significance, since one sees from the proof that restricting to slopes in any positive-length interval yields the same distribution of gaps. If $(X,\omega)$ is periodic for the horocycle flow, then the set of gaps repeats, so this restriction can be removed. In this case, the derivation of the gap distribution uses the equidistribution of long closed horocycles (see \cite{Ath13}), and one obtains the same distribution using either definition of $\bbS_R(X,\omega)$. Finally, we note that the derivation that follows shows that every surface in the $\slr$-orbit of $(X,\omega)$ has the same gap distribution (see also \cite{KSW}) and that any Veech surface has elements in its orbit which are periodic for the horocycle flow, so if we are solely interested in Veech surfaces, we could instead take this approach. Regardless, we use the given definition because it applies more generally and is common in the literature.}
\[
\bbS_R(X,\omega) := \{\slp(v) \mid v\in\Lambda, \ 0\leq \Im(v)\leq \Re(v) \leq R\}
\]
and let $N(R):=|\bbS_R(X,\omega)|$ be the number of distinct slopes in this set. From now on, we consider $\bbS_R(X,\omega)$ as an ordered set:
\[
\bbS_R(X,\omega) = \{0\leq  s^R_1 < \cdots < s^R_{N(R)}\}.
\] 
We are interested in the distribution of gaps between elements in this set as $R\to\infty$, however, it follows from \cite{Veech98, Vorobets} that (for almost every translation surface and for all Veech surfaces) the set of slopes becomes dense in $\R$ as $R\to\infty$, so the absolute gap size goes to $0$. Since we know from \cite {Masur88,Masur90} 
that $N(R)$ grows quadratically, it makes sense to instead consider the \emph{renormalized} gap set
\begin{align*}
\bbG_R(X,\omega) := \left\{R^2(s^R_{i+1} - s^R_i)  \mmid s_i\in \bbS_R(X,\omega), \ i=1,\dots,N(R)-1\right\}.
\end{align*}
As before, we will simply denote $\bbS_R(X,\omega)$ by $\bbS_R$ and $\bbG_R(X,\omega)$ by $\bbG_R$ when the surface is clear from context. We are interested in finding a probability density function $f:\R\to[0,\infty)$ such that
\[
\lim_{R\to\infty}\frac{\left|\bbG_R(X,\omega)\cap [a,b]\right|}{N(R)-1} = \int_a^b f(x) dx
\]
for any $[a,b]\subseteq \bbR$. To do this, we will adopt the method of \cite{AthCha,AC13}, which transforms the problem of studying slope gaps into a problem about the dynamics of the horocycle flow on the homogeneous space $\slr/\slxo$.

\subsection{Dynamical Reframing}

Define the \emph{horocycle flow} to be the action of the one-parameter subgroup 
\[
H = \left\{ h_s = \begin{pmatrix} 1&0\\-s&1\end{pmatrix} \mmid s\in\R \right\}
\]
which acts as a downward vertical shear on the plane. 
Observe that $h_s$ acts by translation on slopes, namely
\[
\slp(h_s v) = \slp(v) - s.
\]
In particular, this means that it preserves gaps between slopes and that the slope of a vector is precisely the time it takes that vector to hit the $x$-axis under the horocycle flow. The horocycle flow acting on the space of translation surfaces naturally restricts to the $\slr$-orbit of $(X,\omega)$, which can be identified with the moduli space $\slr/\slxo\cong\slr\cdot (X,\omega)$, and in this context $H$ acts on $\slr/\slxo$ by left multiplication. 

Recall that a \emph{Poincar\'e section} or \emph{transversal} for a flow is a lower-dimensional subset of the space that almost every orbit of the flow passes through infinitely often in a discrete subset of times. In \cite{Ath13}, Athreya shows that the subset of surfaces with a short horizontal saddle connection of length less than or equal to one, given by
\[
\Omega := \{g\slxo \mid \Lambda(g\cdot(X,\omega))\cap(0,1]\neq\emptyset\},
\]
is a Poincair\'e section for the horocycle flow on $\slr/\slxo$, where we think of $(0,1]\in\bbC$ as a subset of the real axis. 
Define the \emph{first return time function} $\mathcal{R}:\Omega\to\bbR^+$ to be the time it takes a point $x\in \Omega$ to return to the set $\Omega$ under the horocycle flow, i.e.,
\[
\mathcal{R}(x) = \min\{s>0\mid h_s\cdot x\in\Omega\}.
\]
Define the \emph{return map} $T:\Omega\to\Omega$ to be the next intersection point of $x\in\Omega$ with the Poincar\'e section under the horocycle flow:
\[
T(x) = h_{\mathcal{R}(x)} \cdot x.
\]

Now observe that the diagonal matrix\footnote{This is an unusual parametrization of the geodesic flow, which would typically be parametrized $a_t=\left(\begin{smallmatrix}e^{t/2}&0\\0&e^{-t/2}\end{smallmatrix}\right)$, but this parametrization is more convenient for our purposes.}
\[
a_{R} := \begin{pmatrix}R^{-1}&0\\0&R\end{pmatrix}\in\slr
\]
contracts the plane horizontally by a factor of $R^{-1}$ and stretches the plane vertically by a factor of $R$, so it scales slopes by a factor of $R^2$, that is
\[
\slp(a_R v) = R^2\slp(v)
\]
for any $v\in\bbC$. Then 
\begin{align*}
R^2 \bbS_R(X,\omega) &=\{0\leq R^2s^R_1<\cdots<R^2s^R_{N(R)}\}\\
&= \left\{\slp(a_Rv) \mmid v\in\Lambda(X,\omega), \ 0\leq \Im(v)\leq \Re(v)\leq R\right\}\\
&= \left\{\slp(v) \mmid v\in a_R\Lambda(X,\omega), \ 0\leq \Im(v)\leq R^2\Re(v)\leq R^2\right\}\\
&= \left\{\slp(v) \mmid v\in \Lambda(a_R\cdot (X,\omega)), \ 0\leq \Im(v)\leq R^2\Re(v)\leq R^2\right\}\\
&= \left\{s \mmid h_{s}\cdot a_R\slxo\in\Omega, \ s\in [0,R^2] \right\}.
\end{align*}
In other words, the set of slopes in $\bbS_R(X,\omega)$ multiplied by $R^2$ is the same as the times at which the surface $a_R\cdot(X,\omega)$ intersects the transversal $\Omega$ under the horocycle flow when flowed up to time $R^2$. Thus the renormalized gaps are precisely the times between returns of $a_R\slxo$ to $\Omega$, i.e.,
\begin{align*}
\bbG_R(X,\omega) 
&= \{\mathcal{R}(T^{i-1}(x_0^R)) \mid i=1,\dots,N(R)-1\}
\end{align*}
where $x_0^R:= h_{R^2s_1^R}\cdot a_R\slxo =  a_R\cdot h_{s_1^R}\slxo$ is the first point at which the horocycle orbit of $a_R\slxo$ intersects $\Omega$. This allows us to write our desired distribution as a limit of Birkhoff sums of the form
\begin{align*}
\lim_{R\to\infty} \frac{|\bbG_R(X,\omega)\cap[a,b]|}{N(R)-1} 
&= \lim_{R\to\infty} \frac{1}{N(R)-1}\sum_{i=1}^{N(R)-1} \chi_{\mathcal{R}^{-1}([a,b])}(T^{i-1}(x_0^R))
\end{align*}
where
\[
\chi_{\mathcal{R}^{-1}([a,b])}(x) = \begin{cases}
    1,&\mathcal{R}(x)\in[a,b]\\
    0,&\mathcal{R}(x)\notin[a,b]
\end{cases}
\]
is the indicator function on $\mathcal{R}^{-1}([a,b])\subseteq \Omega$. The following lemma states that this limit in fact converges to the unique invariant probability measure for the return map on $\Omega$, which will allow us to find the probability density function $f$. Similar statements can be found in \cite{AC13,ACL,Ath13}, but we include the proof for the sake of completeness. We also remark that the following statement and proof do not differentiate between points that are periodic versus non-periodic for the horocycle flow, nor whether the surface $(X,\omega)$ has a horizontal saddle connection.
\begin{lem}
Suppose $(X,\omega)$ is a Veech surface with $\bbS_R$, $N(R)$, $\Omega$, $T$, $\mathcal{R}$, and $x_0^R$ defined as above. Then for any bounded, measurable function $f:\Omega\to\bbR$,
\[
\lim_{R\to\infty} \frac{1}{N(R)-1} \sum_{i=1}^{N(R)-1} f(T^{i-1}(x^R_0)) = \int_{\Omega} f dm.
\]
where $m$ is the unique ergodic probability measure for $T$ supported on $\Omega$.
\end{lem}
\begin{proof}
Let $\Gamma = \slxo$. We first note that $\slr/\Gamma$ is a suspension space over $\Omega$ with roof function given by $\mathcal{R}$
and that the probability Haar measure $\mu$ on $\slr/\Gamma$ decomposes as $d\mu = C ds dm$ for some constant $C$, so for any integrable $f:\slr/\Gamma\to\bbR$, we have 
\begin{align}
\int fd\mu 
&= C\int_\Omega\int_0^{\mathcal{R}(x)} f(h_sx) ds dm(x). \label{eq:mudecomp}
\end{align}
It follows from results of Kleinbock-Margulis (see, e.g., Proposition 2.2.1 in \cite{KMbdd}) that expanding translates of pieces of horocycles equiduistribute in $\slr/\Gamma$. In particular, 
\begin{align}
\lim_{R\to\infty} \int_0^1 f(a_R h_s\Gamma)ds &= \int fd\mu. \label{eq:expandingtranslates}
\end{align}
 But we may also write
\begin{align*}
 &\int_0^1 f(a_R h_s\Gamma)ds\\
 &=  \int_0^{1} f(a_Rh_s a_R^{-1}  a_R\Gamma)ds\\
 &=\frac{1}{R^2} \int_0^{R^2} f(h_s  a_R\Gamma)ds\\
 &= \frac{1}{R^2} \left(\int_0^{R^2s_1^R} f(h_s  a_R\Gamma)ds+\int_{R^2s_1^R}^{R^2s_2^R}f(h_s  a_R\Gamma)ds+\cdots +\int_{R^2s_{N(R)}^R}^{R^2}f(h_s  a_R\Gamma)ds\right) \\
 &= \frac{1}{R^2} \left(\int_0^{R^2s_1^R} f(h_s a_R\Gamma)ds+\sum_{i=1}^{N(R)-1} \int_0^{\mathcal{R}(T^{i-1}(x_0^R))}f(h_sT^{i-1}(x_0^R))ds+\int_{R^2s_{N(R)}^R}^{R^2}f(h_s  a_R\Gamma)ds\right).
\end{align*}
Now note that 
\begin{align*}
    \left|\frac{1}{R^2}\int_0^{R^2s_1^R} f(h_s  a_R\Gamma)ds\right|\leq \frac{R^2s_1^R}{R^2}\ \|f\|_\infty=s_1^R\|f\|_\infty\to 0
\end{align*}
as $R\to\infty$, where $\|f\|_\infty$ is the sup-norm of $f$. This is because the slopes in $\bbS_R$ become dense in $[0,1]$ as $R\to\infty$, so $s_1^R\to 0$. Similarly,
\begin{align*}
    \left|\frac{1}{R^2}\int_{R^2s_{N(R)}^R}^{R^2} f(h_s  a_R\Gamma)ds\right|\leq (1-s_{N(R)}^R)\|f\|_\infty\to 0
\end{align*}
as $R\to\infty$, since $s_{N(R)}^R\to 1$. Thus
\begin{align*}
\lim_{R\to\infty} \int_0^1 f(a_R h_s\Gamma)ds = \lim_{R\to\infty} \frac{N(R)-1}{R^2}\left(\frac{1}{N(R)-1} \sum_{i=1}^{N(R)-1} \int_0^{\mathcal{R}(T^{i-1}(x_0^R))}f(h_sT^{i-1}(x_0^R))ds\right).
\end{align*}
We know that $\frac{N(R)-1}{R^2}\to C$ by \cite{Veech89}. Let $\sigma_R$ be the normalized counting measure on the intersections of the orbit of $a_R\Gamma$ with $\Omega$ under the horocycle flow when flowed up to time $R^2$, i.e.
\[
\sigma_R = \frac{1}{N(R)-1} \sum_{i=1}^{N(R)-1} \delta_{T^{i-1}(x_0^R)}.
\]
Then this says that
\begin{align*}
\lim_{R\to\infty} \int_0^1 f(a_R h_s\Gamma)ds = \lim_{R\to\infty}  C\int_\Omega\int_0^{\mathcal{R}(x)}f(h_s x)dsd\sigma_R(x).
\end{align*}
Together with \ref{eq:mudecomp} and \ref{eq:expandingtranslates}, this implies the weak convergence $\sigma_R \to m$, which finishes the proof.
\end{proof}
This lemma means that we can represent our gap distribution by
\begin{align*}
    \lim_{R\to\infty} \frac{|\bbG_R(X,\omega)\cap[a,b]|}{N(R)-1} = m(\{x\in\Omega \mid \mathcal{R}(x)\in [a,b]\}).
\end{align*}
So in order to calculate the distribution function from a practical standpoint, we need to find a convenient way to parametrize the Poincar\'e section $\Omega$, understand the measure $m$ with respect to this parametrization, and determine the return time function $\mathcal{R}$ at every point on $\Omega$. Luckily, the first two tasks are already well understood in the literature, and we now provide an overview of those constructions.

\subsection{The Kumanduri-Sanchez-Wang Algorithm}
\label{sec:KSWalg}

We build on work of Kumanduri, Sanchez, and Wang in \cite{KSW} parametrizing the transversal $\Omega$ for an arbitrary Veech surface, which itself modifies the algorithm developed by Uyanik and Work in \cite{UW}. 
In this section, we introduce their notation and summarize their results, but a more thorough treatment can be found in the original papers.

Let $(X,\omega)$ be a Veech surface and suppose its Veech group $\slxo$ has $n$ cusps, corresponding to the conjugacy classes of maximal parabolic subgroups of $\slxo$. Let $\Gamma_1,\dots,\Gamma_n$ be representatives of these $n$ conjugacy classes. 
Each of these parabolic subgroups gives rise to a different component of the parametrization of $\Omega$, but exactly what these pieces look like depends on whether $-\Id\in\slxo$.

\subsubsection{Case 1}
First, assume $-\Id\in\slxo$. Then $\Gamma_i\cong \bbZ\oplus\bbZ/2\bbZ$ for each $i\in\{1,\dots,n\}$, and for each $i$ we can choose a parabolic generator $P_i$ for the infinite cyclic factor of $\Gamma_i$ that has eigenvalue $1$. Let $w_i$ be a shortest holonomy vector in the eigenspace of $P_i$. Then the set of all holonomy vectors for $(X,\omega)$ decomposes into a finite union of $\slxo$-orbits of these vectors, i.e. 
\[
\Lambda= \bigcup_{i=1}^n \slxo w_i.
\]
 After possibly replacing $P_i$ with its inverse, it is possible to find a matrix $C_i\in\slr$ with the properties
 \begin{itemize}
     \item $C_iP_iC_i^{-1} = 
     \begin{pmatrix}1&\alpha_i\\0&1\end{pmatrix}$ for some $\alpha_i>0$, and
     \item $C_iw_i = \begin{pmatrix}1\\0\end{pmatrix}$, i.e., $C_i\cdot (X,\omega)$ has shortest horizontal saddle connection of length $1$.
 \end{itemize}
For each $i$, let $(x^i_{0},y^i_{0})$ be the holonomy vector for $C_i(X,\omega)$ satisfying the following two properties in order:
    \begin{itemize}
        \item $y^i_0 > 0$ is the smallest positive $y$-component among all holonomy vectors on $C_i(X,\omega)$, and
        \item $x^i_0 > 0$ is the smallest positive $x$-component among all holonomy vectors on $C_i(X,\omega)$ with $y$-component equal to $y^i_0$.
    \end{itemize}
Define 
\[
M_{a,b} = \begin{pmatrix}a&b\\0&a^{-1}\end{pmatrix}.
\]
Then $\Omega = \bigsqcup_{i=1}^n \Omega_i$, where
\[
\Omega_i = \left\{M_{a,b}C_i\cdot \slxo \mmid (a,b)\in T_i \right\}
\]
and where $T_i\subseteq\bbR^2$ is the triangle 
\[
T_i = \left\{(a,b)\in\bbR^2 \mmid 0 < a \leq 1, \; \left(-\frac{x^i_0}{y^i_0} - \alpha_i \right) a + \frac{1}{y^i_0} < b \leq -\frac{x^i_0}{y^i_0} a + \frac{1}{y^i_0}
    \right\}.
\]
Thus, $T_i$ parametrizes the space $\Omega_i$, and the unique invariant ergodic probability measure $dm$ on $\Omega$ is then just the appropriately scaled Lebesgue measure $dadb$ on $\sqcup_{i=1}^n T_i$ under this parametrization. We will often abuse notation and write $\Omega_i$ for either $\Omega_i\subseteq \slr/\slxo$ or $T_i\subseteq \bbR^2$, depending on the context.

Let $\Lambda_i=C_i\cdot \Lambda$ denote the set of holonomy vectors on $C_i\cdot(X,\omega)$. For a point $M_{a,b}C_i\cdot\slxo\in\Omega_i$, its return time will be given by the slope of the first holonomy vector on $M_{a,b}C_i\cdot(X,\omega)$ to hit the $x$-axis in the interval $(0,1]$ when flowed forward by the horocycle flow. 
If $v=(x,y)\in \Lambda_i$ is the holonomy vector on $C_i\cdot(X,\omega)$ such that $M_{a,b}v$ is this 
vector, then in the coordinates of $T_i$, the return time function is given by
\[
\mathcal{R}(a,b) = \slp(M_{a,b}v) = \frac{1}{a(ax+by)}.
\]
Thus, in order to calculate the slope gap distribution for $(X,\omega)$, we need to determine this vector for every $(a,b)\in\Omega$. The following definitions are taken from \cite{KSW}.
\begin{defn}
    A vector $v\in \Lambda_i$ is the \emph{winning vector} or \emph{winner} at $(a,b) \in \Omega_i$ if $M_{a,b}v$ is the first holonomy vector on $M_{a,b}C_i\cdot(X,\omega)$ to intersect the $x$-axis in the interval $(0,1]$ when flowed forward by the horocycle flow. Equivalently, $M_{a,b}v$ is the holonomy vector of least positive slope on $M_{a,b}C_i\cdot(X,\omega)$ with $x$-component in $(0,1]$. If there are multiple vectors of least slope, we choose the shortest one as the winning vector.
\end{defn}
\begin{defn}
    A vector $(x,y)\in\Lambda_i$ is a \emph{candidate winning vector} at $(a,b) \in \Omega_i$ if 
    it satisfies $0 < ax + by \leq 1$ and $y > 0$.
\end{defn}
This means that $M_{a,b}\cdot(x,y) = (ax+by,a^{-1}y)$ has positive $y$-coordinate and $x$-coordinate in $(0,1]$, so it will eventually intersect the $x$-axis in the interval $(0,1]$ when flowed forward by the horocycle flow. It will be the winning vector at $(a,b)$ if it is the first such vector to do so.

\begin{defn}
\label{def:KSWcandidacystrips}
    Given a holonomy vector $(x,y)\in\Lambda_i$, define the strip $S_{\Omega_i}(x,y)$ to be the set of points in the plane of $\Omega_i$ where $(x,y)$ could be a candidate winning vector. That is,
    \[
    S_{\Omega_i}(x,y) = \{ (a,b)\in\bbR^2 \; \vert \; 0 < ax + by \leq 1,\; a>0 \}.
    \]
    Similarly, define $S_{\Lambda_i}(a,b)$ to be the strip in the plane of  $\Lambda_i$ containing holonomy vectors for $C_i\cdot (X,\omega)$ which could possibly win at $(a,b)\in\Omega_i$, or
    \[
    S_{\Lambda_i}(a,b) = \{ (x,y)\in\bbR^2 \; \vert \; 0 < ax + by \leq 1,\; y>0 \}.
    \]
\end{defn}
Note that $S_{\Lambda_i}(a,b)$ includes vectors that are not holonomy vectors, and hence are not candidates. Being in $S_{\Lambda_i}(a,b)$ is a necessary but not sufficient condition to be a candidate winning vector at $(a,b)$. Likewise, $S_{\Omega_i}(x,y)$ includes points outside of $\Omega_i$, even though candidate winning vectors are not actually defined for these points. 

Finally, we observe the following fact:
\begin{lem}
    The matrix $M_{a,b}$ with $a>0$ preserves the order of slopes. That is, $\slp(v_1)\leq\slp(v_2)$ if and only if $\slp(M_{a,b}v_1) \leq \slp(M_{a,b}v_2)$, where we consider an undefined slope to be larger than any positive slope and smaller than any negative slope.\footnote{Notice that this means a negative slope is considered larger than a positive slope. In all cases, the ordering on slopes can be succinctly defined by the relation: $(x_1,y_1)$ has smaller slope than $(x_2,y_2)$ if and only if $\frac{x_1}{y_1}<\frac{x_2}{y_2}$.}
\end{lem}
\begin{proof}
Geometrically, the linear transformation $M_{a,b}$ consists of scaling the plane horizontally and vertically by positive quantities, which does not change the order of slopes of vectors, as well as a horizontal shear, which also preserves the order of slopes. It is an exercise to verify this fact algebraically.
\end{proof}
This means that the winning vector at $(a,b)\in\Omega_i$ is simply the shortest candidate winning vector at $(a,b)$ of least slope. 

\subsubsection{Case 2}
Now assume $-\Id\notin\slxo$. This case is similar to Case 1, except now $\Gamma_i\cong \bbZ$ for all $i$, and for each parabolic subgroup, its generator $P_i$ either has eigenvalue $1$ or $-1$. Suppose that $P_1,\dots,P_k$ have eigenvalue $1$ and $P_{k+1},\dots, P_n$ have eigenvalue $-1$, where $1\leq k\leq n$. In this case, the set of holonomy vectors for $(X,\omega)$ decomposes as
\[
\Lambda = \bigcup_{i=1}^n \left(\slxo w_i\cup \slxo (-w_i)\right)
\]
where $w_i$ is a shortest holonomy vector in the eigenspace of $P_i$. 
Then (after possibly replacing $P_i$ with its inverse) for each $i\in\{1,\dots,k\}$ there is a $C_i\in\slr$ with the same properties as in Case 1, and for $i\in\{k+1,\dots,n\}$, there is $C_i\in\slr$ such that 
\[
C_iP_iC_i^{-1} = \begin{pmatrix}-1&\alpha_i\\0&-1\end{pmatrix}
\]
for some $\alpha_1>0$, and such that $C_i\cdot (X,\omega)$ has shortest horizontal saddle connection of length $1$. 
Define $(x_0^i,y_0^i)\in C_i\cdot\Lambda$ as before. 
Then $\Omega = \bigsqcup_{i=1}^n \Omega_i$, where
\[
\Omega_i = \left\{M_{a,b}C_i\cdot \slxo \mmid (a,b)\in T_i \right\},
\]
but now the set $T_i$ depends on whether $P_i$ has eigenvalue $1$ or $-1$. If $i\in\{1,\dots,k\}$, then 
\[
T_i = T_{i,1}\cup T_{i,2}
\]
where $T_{i,1}$ is defined in the same way as $T_i$ from Case 1, and $T_{i,2} = \{(-a,-b)\mid (a,b)\in T_{i,1}\}$. If $i\in\{k+1,\dots,n\}$, then 
\[
T_i = \left\{(a,b)\in\bbR^2 \mmid 0 < a \leq 1, \; \left(-\frac{x^i_0}{y^i_0} - 2\alpha_i \right) a + \frac{1}{y^i_0} < b \leq -\frac{x^i_0}{y^i_0} a + \frac{1}{y^i_0}
    \right\}.
\]
In both cases, the measure $m$ is still the scaled Lebesgue measure under this parametrization. We define winning vectors and candidate winning vectors in the same way as before, and the formula for $\mathcal{R}$ in coordinates is also the same. We note that if $v$ wins at $(a,b)\in T_{i,1}$ for $i\in\{1,\dots,k\}$, then $-v$ wins at $(-a,-b)\in T_{i,2}$. So in all cases, the problem of determining the winning vectors reduces to finding the winning vectors on the triangle
\begin{align}
\left\{(a,b)\in\bbR^2 \mmid 0 < a \leq 1, \; \left(-\frac{x^i_0}{y^i_0} - n\alpha_i \right) a + \frac{1}{y^i_0} < b \leq -\frac{x^i_0}{y^i_0} a + \frac{1}{y^i_0}
    \right\}\subseteq\bbR^2
    \label{eq:KSWtriangle}
\end{align}
for $n=1$ or $2$.

\subsubsection{Computing the Probability Density Function}
\label{sec:finding-pdf}

In \cite{KSW}, the authors show that the winning vectors partition each transversal component into a finite set of polygonal regions. Once the transversal has been parameterized and the return time function has been determined at each point in $\Omega$, finding the cumulative distribution function $F$ amounts to finding the area within each polygon swept out by level sets of the hyperbola that defines the return time function on that region. Explicitly,
\begin{align*}
    F(t) &= \sum_{i=1}^n m(\{(a,b)\in T_i \mid 0<\mathcal{R}(a,b)\leq t\})\\
    &=\sum_{i=1}^n \sum_{j=1}^{k_i} m(\{(a,b)\in \Delta_i(v_j) \mid 0<\mathcal{R}_j(a,b)\leq t\})
\end{align*}
where $v_1,\dots, v_{k_i}$ are the winning vectors on $\Omega_i$, $\Delta_i(v_j)$ is the polygonal subset of $T_i$ on which $v_j$ wins, and $\mathcal{R}_j$ is the return time function determined by $v_j$. The probability density function $f$ is then simply given by $f(t) = F'(t)$. The measure $m$ is the Lebesgue measure normalized so that $\int_\R f dm = 1$.

\section{Finding Winning Vectors}\label{sec:winners}

Let $(X,\omega)$ be a Veech surface. Assume we know $(x^i_0,y^i_0)$, $P_i$, $C_i$, and $\alpha_i$ for that surface, as defined in Section~\ref{sec:KSWalg}. From now on, we will focus on a single value of $i$, so we suppress the index for the remainder of this document. 

\subsection{Reparameterization}

We will find it convenient to work with the following reparameterization of the transversal:
\begin{defn}
    Define $\Omega$ as the following triangle in $\bbR^2$:
    $$\Omega = \left\{ (a,b) \in \bbR^2 \mmid 0 < b \leq 1, \; \frac{x_0}{y_0} b - \frac{1}{y_0} \leq a < \; \right(\frac{x_0}{y_0} + n\alpha\left) b -\frac{1}{y_0} \right\},$$
    where $n=1$ or $2$, depending on whether 
    the parabolic generator $P$ corresponding to this piece of the transversal has eigenvalue $1$ or $-1$ (cf. Section~\ref{sec:KSWalg}).
\end{defn}

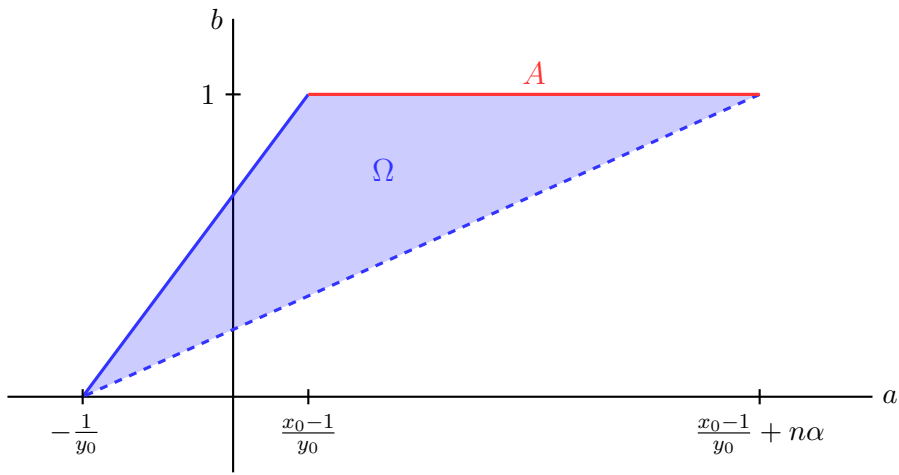
\begin{figure}
    \begin{tikzpicture}
        \draw[thick] (-3,0)--(8.5,0);
        \draw[thick] (0,-1)--(0,5);
        \fill[fill=blue, opacity=0.2] (-2,0)--(1,4)--(7,4);
        \node[blue!80] at (2,3) {\large $\Omega$};
        \draw[very thick, blue!80] (-2,0)--(1,4);
        \draw[very thick, blue!80, dashed] (7,4)--(-2,0);
        \draw[very thick, red!80] (1,4)--(7,4);
        \node[above,red!80] at (4,4) {\large $A$};
        \node[right] at (8.5,0) {$a$};
        \node[left] at (0,5) {$b$};
        \draw[thick] (-2,-0.1)--(-2,0.1);
        \node[below] at (-2.1,-0.1) {$-\frac{1}{y_0}$};
        \draw[thick] (1,-0.1)--(1,0.1);
        \node[below] at (1,-0.1) {$\frac{x_0-1}{y_0}$};
        \draw[thick] (7,-0.1)--(7,0.1);
        \node[below] at (7,-0.1) {$\frac{x_0-1}{y_0}+n\alpha$};
        \draw[thick] (-0.1,4)--(0.1,4);
        \node[left] at (-0.1,4) {$1$};
    \end{tikzpicture}
    \caption{Our parameterization of the transversal $\Omega$. Its top edge, $A$, is shown in red.}
    \label{fig:Omega}
\end{figure}

Compared to the parameterization given in \eqref{eq:KSWtriangle}, which is taken from \cite{KSW}, we have rotated $\Omega$ counterclockwise by $\pi/2$ radians; see Figure~\ref{fig:Omega}. This rotation corresponds to replacing $M_{a,b}$ with $M_{b,-a}$, so that now a point $(a,b)\in\Omega$ represents the translation surface
\[
M_{b,-a}C\cdot (X,\omega) = \begin{pmatrix}b&-a\\0&1/b\end{pmatrix}C\cdot (X,\omega)
\]
which has a short horizontal saddle connection of length $b\leq 1$. Now a vector $(x,y)\in \Lambda$ is a candidate winning vector at $(a,b)$ if $M_{b,-a}\cdot(x,y) = (bx-ay, b^{-1}y)$ has positive $y$-coordinate and $x$-coordinate in $(0,1]$. 
This gives rise to the following rotated re-definitions of the candidacy strips from Definition~\ref{def:KSWcandidacystrips}:
\begin{defn}\label{def:ourcandidacystrips}
    Given $(x,y)\in\Lambda$, define the strip $S_{\Omega}(x,y)$ in the plane of $\Omega$ by
    \[
    S_{\Omega}(x,y) = \{ (a,b)\in\bbR^2 \; \vert \; 0 < bx-ay \leq 1,\; b>0 \}.
    \]
    Similarly, given $(a,b)\in\Omega$, define $S_{\Lambda}(a,b)$ in the plane of $\Lambda$ by
    \[
    S_{\Lambda}(a,b) = \{ (x,y)\in\bbR^2 \; \vert \; 0 < bx-ay \leq 1,\; y>0 \}.
    \]
\end{defn}
Figure~\ref{fig:candidacy-strips} demonstrates why we chose this reparameterization of $\Omega$: The rotation makes it so that a point $(a,b)$ or vector $(x,y)$ lies on the boundary of the strip it defines, rather than lying at a location with no clear visual relation to its strip (although we note that we are viewing them in the ``wrong" plane, as $(x,y)\in\Lambda$ defines a strip in the plane of $\Omega$, and vice versa). In particular, holonomy vectors and points in $\Omega$ now have the same slope as the strips they define, which may make them more intuitive to work with. 

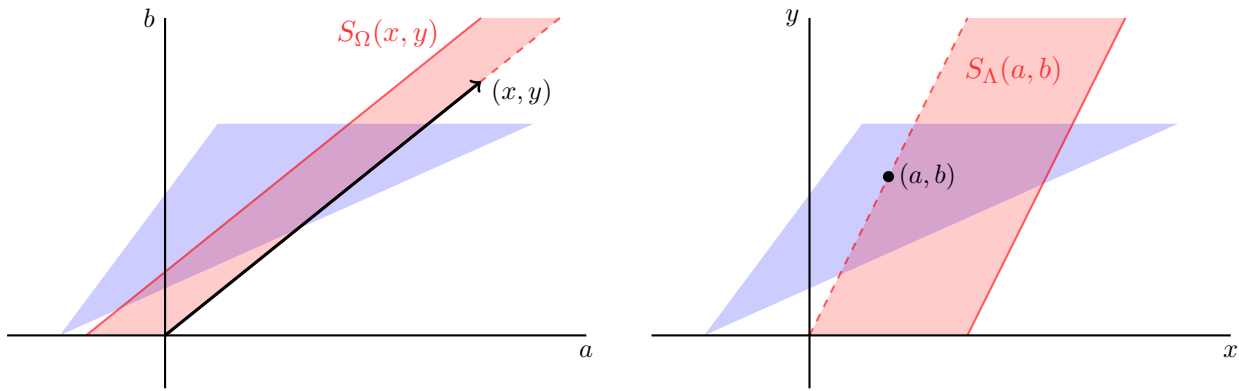
\begin{figure}
    \begin{tikzpicture}[scale=0.7]
        \fill[fill=blue, opacity=0.2] (-2,0)--(1,4)--(7,4);
        \fill[fill=red, opacity=0.2] (-1.5,0)--(6,6)--(7.5,6)--(0,0);
        \draw[thick,red,opacity=0.6] (-1.5,0)--(6,6);
        \draw[thick,dashed,red,opacity=0.6] (7.5,6)--(0,0);
        \draw[->,very thick] (0,0)--(6,4.8);
        \node[right] at (6,4.6) {\small $(x,y)$};
        \node[left,red!80] at (5.4,5.7) {$S_\Omega(x,y)$};
        \node[below] at (8,0) {\small $a$};
        \node[left] at (0,6) {\small $b$};
        \draw[thick] (-3,0)--(8,0);
        \draw[thick] (0,-1)--(0,6);
    \end{tikzpicture}
    \hspace{1em}
    \begin{tikzpicture}[scale=0.7]
        \fill[fill=blue, opacity=0.2] (-2,0)--(1,4)--(7,4);
        \fill[fill=red, opacity=0.2] (0,0)--(3,6)--(6,6)--(3,0);
        \draw[thick,dashed,red,opacity=0.6] (0,0)--(3,6);
        \draw[thick,red,opacity=0.6] (6,6)--(3,0);
        \node[left,red!80] at (5,5) {$S_\Lambda(a,b)$};
        \fill[black] (1.5,3) circle (3pt);
        \node[right] at (1.5,3) {\small $(a,b)$};
        \node[below] at (8,0) {\small $x$};
        \node[left] at (0,6) {\small $y$};
        \draw[thick] (-3,0)--(8,0);
        \draw[thick] (0,-1)--(0,6);
    \end{tikzpicture}
    \caption{Candidacy strips $S_\Omega(a,b)$ and $S_\Lambda(x,y)$ under our reparameterization. Observe that $(x,y)$ and $(a,b)$ lie on the boundary of, and have the same slopes as, the strips they define.}
    \label{fig:candidacy-strips}
\end{figure}

If a vector $v=(x,y)\in\Lambda$ wins at the point $(a,b)\in\Omega$, then the return time function at that point under this new parameterization is given by 
\[
\mathcal{R}(a,b) = \slp(M_{b,-a}v) = \frac{y}{b(bx-ay)}.
\]

\subsection{Winning Vectors Along the Top Edge of $\Omega$}

We now define a subset of $\Omega$ that will be of particular interest to us:
\begin{defn}
    Let $A$ denote the top edge of $\Omega$. Explicitly,
    \begin{align*}
        A &= \Omega \cap \{b = 1\} \\
        &= \left\{ (a,1) \in \bbR^2 \mmid \frac{x_0 - 1}{y_0} \leq a < \frac{x_0 - 1}{y_0} + n\alpha \right\}.
    \end{align*}
    Instead of a horizontal line segment in $\bbR^2$, we will usually think of $A$ as an interval in $\bbR$:
    $$A = \left[\frac{x_0 - 1}{y_0}, \frac{x_0 - 1}{y_0} + n\alpha \right).$$
    By abuse of notation, we will write $A$ to mean either the segment or the interval depending on context.
\end{defn}

This subset is significant because it turns out that any vector that wins anywhere on $\Omega$ must win somewhere along the top edge. This fact was proved as Lemma A.1 of \cite{Previous_SUMRY} in the specific context of the regular $2n$-gon surface, but we show that this phenomenon holds more generally for any Veech surface parameterized in this way.\footnote{Since writing this article, we learned of unpublished work by Michael Beers and Gabriela Brown, who independently proved a similar result in the context of square-tiled surfaces and developed efficient code for computing the slope gap distributions of square-tiled surfaces \cite{BB-STSs}.} In fact, we can say more:

\begin{lem}
    Let $(x,y)\in \Lambda$ be a vector that wins on $\Omega$. Then $(x,y)$ wins at the point $\left(\frac{x - 1}{y}, 1\right)\in A$ where the left edge of the strip $S_\Omega(x,y)$ meets the top edge of $\Omega$.
    \label{lem:A_winners_are_Omega_winners}
 \end{lem}   

\begin{proof}
The left edge of the candidacy strip $S_\Omega(x,y)$ intersects the line $b=1$ at the point $\left(\frac{x - 1}{y}, 1\right)$. We will start by showing that this point really is in $A$, i.e.\ that
\[
\frac{x_0-1}{y_0}\leq \frac{x-1}{y}<\frac{x_0-1}{y_0}+n\alpha.
\]
Observe that the $x$-intercept of the left edge of $S_\Omega(x,y)$ is given by $-\frac{1}{y}$. By construction, $y_0\leq y$ for any holonomy vector $(x,y)\in \Lambda$, which means that $-\frac{1}{y_0}\leq -\frac{1}{y}$. Then if $\frac{x_0-1}{y_0}+n\alpha\leq\frac{x - 1}{y}$, it would imply that the strip $S_\Omega(x,y)$ lies entirely below the bottom edge defining $\Omega$, so $S_\Omega(x,y)\cap\Omega=\emptyset$ (see Figure~\ref{fig:left-edge-in-A-1}). This would imply $(x,y)$ cannot win at any point on $\Omega$.

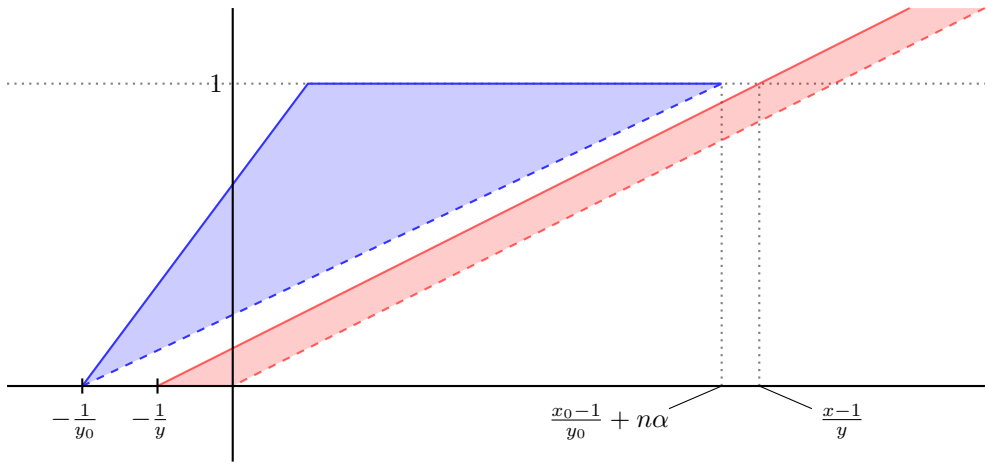
\begin{figure}
    \begin{tikzpicture}
        \draw[gray,dotted,thick] (-3,4)--(10,4);
        \draw[gray,dotted,thick] (6.5,0)--(6.5,4);
        \draw[gray,dotted,thick] (7,0)--(7,4);
        \fill[fill=blue, opacity=0.2] (-2,0)--(1,4)--(6.5,4);
        \draw[thick,blue!80] (-2,0)--(1,4)--(6.5,4);
        \draw[thick,dashed,blue!80] (-2,0)--(6.5,4);
        \fill[fill=red, opacity=0.2] (-1,0)--(9,5)--(10,5)--(0,0);
        \draw[thick,red,opacity=0.6] (-1,0)--(9,5);
        \draw[thick,dashed,red,opacity=0.6] (10,5)--(0,0);
        \draw[thick] (-3,0)--(10,0);
        \draw[thick] (0,-1)--(0,5);
        \draw[thick] (-2,-0.1)--(-2,0.1);
        \draw[thick] (-1,-0.1)--(-1,0.1);
        \node[below] at (-1.1,-0.1) {\small $-\frac{1}{y}$};
        \node[below] at (-2.1,-0.1) {\small $-\frac{1}{y_0}$};
        \node[below] at (5,-0.1) {\small $\frac{x_0-1}{y_0}+n\alpha$};
        \draw (5.8,-0.3)--(6.5,0);
        \draw (7,0)--(7.7,-0.3);
        \node[below] at (8.1,-0.1) {\small $\frac{x-1}{y}$};
        \node[left] at (0,4) {\small $1$};
    \end{tikzpicture}
    \caption{If $\frac{x_0-1}{y_0}+n\alpha\leq\frac{x - 1}{y}$, then $(x,y)$ is not a candidate winning vector (hence not a winner) at any point on $\Omega$.}
    \label{fig:left-edge-in-A-1}
\end{figure}

On the other hand, if $\frac{x-1}{y} < \frac{x_0-1}{y_0}$, then the left edge of $S_\Omega(x,y)$ falls to the left of the left edge of $S_\Omega(x_0,y_0)$, which coincides with the left edge of $\Omega$ by construction. But since the right edge of both strips starts at the origin, this implies that $S_\Omega(x,y)\cap\Omega\subseteq S_\Omega(x_0,y_0)\cap\Omega$ (see Figure~\ref{fig:left-edge-in-A-2}). Moreover, the assumption that $\frac{x-1}{y} < \frac{x_0-1}{y_0}$, together with $y_0\leq y$, implies that the slope of $(x,y)$ is greater than that of $(x_0,y_0)$ (where we maintain the convention that an undefined slope is larger than any positive slope and smaller than any negative slope). Therefore $(x_0,y_0)$ would win over $(x,y)$ at any point in $\Omega$ where $(x,y)$ is a candidate winning vector. Thus if $(x,y)$ wins at some point in $\Omega$, we must have 
\[
\frac{x_0-1}{y_0}\leq \frac{x-1}{y}<\frac{x_0-1}{y_0}+n\alpha
\]
as desired.

\begin{figure}
    \begin{tikzpicture}
        \draw[gray,dotted,thick] (-3,4)--(10,4);
        \draw[gray,dotted,thick] (1,0)--(1,4);
        \draw[gray,dotted,thick] (3,0)--(3,4);
        \fill[fill=blue, opacity=0.2] (-2,0)--(3,4)--(9,4);
        \draw[thick,blue!80] (-2,0)--(3,4)--(9,4);
        \draw[thick,dashed,blue!80] (-2,0)--(9,4);
        \fill[fill=yellow, opacity=0.2] (-2,0)--(4.2,5)--(6.2,5)--(0,0);
        \draw[thick,orange,opacity=0.3] (-2,0)--(4.2,5);
        \draw[thick,dashed,orange,opacity=0.3] (6.2,5)--(0,0);
        \fill[fill=red, opacity=0.2] (-1,0)--(1.5,5)--(2.5,5)--(0,0);
        \draw[thick,red,opacity=0.6] (-1,0)--(1.5,5);
        \draw[thick,dashed,red,opacity=0.6] (2.5,5)--(0,0);
        \draw[thick] (-3,0)--(10,0);
        \draw[thick] (0,-1)--(0,5);
        \draw[thick] (-2,-0.1)--(-2,0.1);
        \draw[thick] (-1,-0.1)--(-1,0.1);
        \draw[thick] (1,-0.1)--(1,0.1);
        \draw[thick] (3,-0.1)--(3,0.1);
        \node[below] at (-1.1,-0.1) {\small $-\frac{1}{y}$};
        \node[below] at (-2.1,-0.1) {\small $-\frac{1}{y_0}$};
        \node[below] at (3,-0.1) {\small $\frac{x_0-1}{y_0}$};
        \node[below] at (1,-0.1) {\small $\frac{x-1}{y}$};
        \node[left] at (0,4) {\small $1$};
    \end{tikzpicture}
    \caption{The strip $S_\Omega(x,y)$ is shown in red and the strip $S_\Omega(x_0,y_0)$ is shown in yellow. If $\frac{x-1}{y} < \frac{x_0-1}{y_0}$, then $(x_0,y_0)$ would beat $(x,y)$ wherever $(x,y)$ is a candidate to win on $\Omega$, so $(x,y)$ cannot win at any point on $\Omega$.}
    \label{fig:left-edge-in-A-2}
\end{figure}
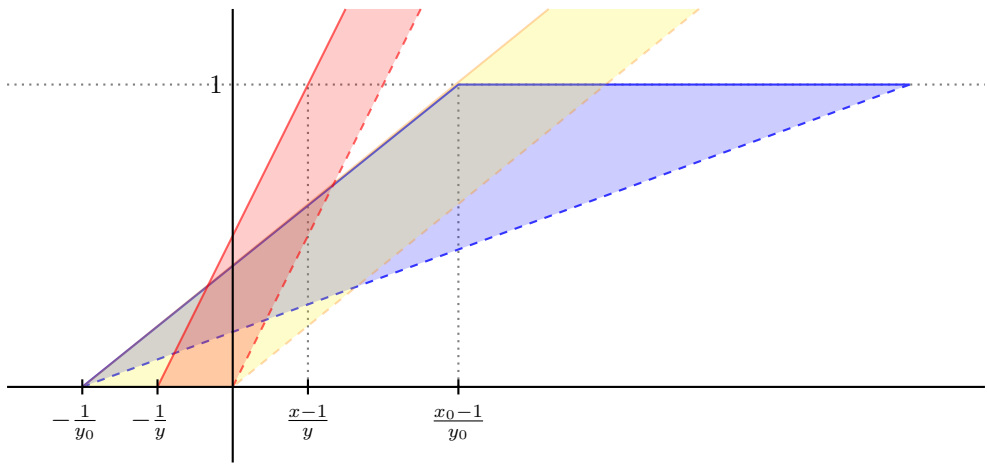

Now suppose for the sake of contradiction that $(x,y)$ wins at $(a,b)\in \Omega$, but another vector $(u,v)$ wins at the point $\left(\frac{x - 1}{y}, 1\right)\in A$. 
We know $\left(\frac{x - 1}{y}, 1\right)\in S_\Omega(x,y)$, so $(x,y)$ is a candidate winning vector at $\left(\frac{x - 1}{y}, 1\right)$. This means that the slope of $(u,v)$ must be no greater than that of $(x,y)$ in order to win, i.e.
\begin{equation}
\frac{u}{v} \geq \frac{x}{y} 
\label{eq:biggerslope}
\end{equation} 
Additionally, if the slopes are equal, then $(u,v)$ must be shorter than $(x,y)$ in order to win. 

Since $(u,v)$ wins at $\left(\frac{x - 1}{y}, 1\right)$, we must also have that $\left(\frac{x - 1}{y}, 1\right) \in S_{\Omega}(u, v)$. Thus $$0 < u - \left(\frac{x - 1}{y}\right)v \le 1,$$ which means that $$0 \leq \frac{u}{v} - \frac{x}{y} \le \frac{1}{v} - \frac{1}{y} $$ where the first inequality is determined by \eqref{eq:biggerslope}. Then, for every $b \in (0,1]$, we get $$0 \leq b \left(\frac{u}{v} - \frac{x}{y}\right) \le \frac{1}{v} - \frac{1}{y} $$ meaning that 
\begin{equation}
    b \left(\frac{x}{y}\right) - \frac{1}{v} \leq b \left(\frac{u}{v}\right) - \frac{1}{v} \le b \left(\frac{x}{y}\right) - \frac{1}{y}.
     \label{eq:criticalinequality}
\end{equation} 
Since we assumed that $(x,y)$ wins at $(a,b)\in\Omega$, we know that $(a,b) \in  S_{\Omega}(x, y) \cap \Omega$. This means that $0 < bx - ay \le 1$ and $b \in (0,1]$, which implies $$a \ge b\left(\frac{x}{y}\right) - \frac{1}{y}\ge b\left(\frac{u}{v}\right) - \frac{1}{v}\geq b \left(\frac{x}{y}\right) - \frac{1}{v}$$ where we have used the inequality in \eqref{eq:criticalinequality}. From this we can derive that $$v\left(b \left(\frac{x}{y}\right) - a\right) \leq bu - av \le 1.$$ Since both $v>0$ and $b \left(\frac{x}{y}\right) - a>0$ (from $0 < bx - ay \leq 1$), this implies $0<bu-av\leq 1$, so $(a,b) \in  S_{\Omega}(u, v).$ However, this contradicts our original assumption that $(x,y)$ wins at $(a,b)\in \Omega$, since $(u,v)$ should beat $(x,y)$ at the point $(a,b)$. Hence, such a vector $(u,v)$ cannot exist and $(x,y)$ must win at the point $\left(\frac{x - 1}{y},1\right)$.

\end{proof}

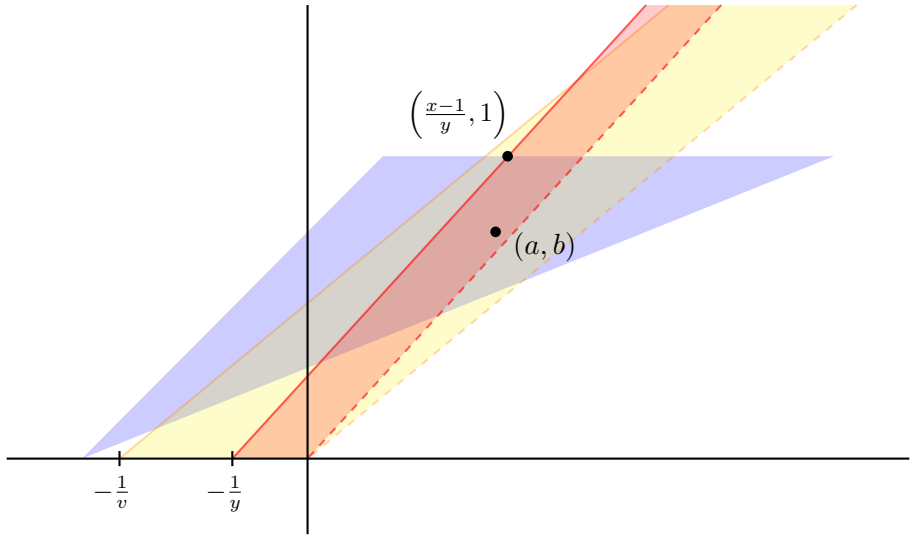
\begin{figure}
    \begin{tikzpicture}
        \fill[fill=blue, opacity=0.2] (-3,0)--(1,4)--(7,4);
        \fill[fill=yellow, opacity=0.2] (-2.5,0)--(4.8,6)--(7.3,6)--(0,0);
        \draw[thick,orange,opacity=0.3] (-2.5,0)--(4.8,6);
        \draw[thick,dashed,orange,opacity=0.3] (7.3,6)--(0,0);
        \fill[fill=red, opacity=0.2] (-1,0)--(4.5,6)--(5.5,6)--(0,0);
        \draw[thick,red,opacity=0.6] (-1,0)--(4.5,6);
        \draw[thick,dashed,red,opacity=0.6] (5.5,6)--(0,0);
        \draw[thick] (-4,0)--(8,0);
        \draw[thick] (0,-1)--(0,6);
        \draw[thick] (-2.5,-0.1)--(-2.5,0.1);
        \draw[thick] (-1,-0.1)--(-1,0.1);
        \node[below] at (-1.1,-0.1) {\small $-\frac{1}{y}$};
        \node[below] at (-2.6,-0.1) {\small $-\frac{1}{v}$};
        \node[above] at (2,4.1) {\small $\left(\frac{x-1}{y},1\right)$};
        \fill[black] (2.66,4) circle (2pt);
         \node[right] at (2.6,2.8) {$(a,b)$};
        \fill[black] (2.5,3) circle (2pt);
    \end{tikzpicture}
    \caption{The strip $S_\Omega(x,y)$ is shown in red and the strip  $S_\Omega(u,v)$ is shown in yellow. Any $S_\Omega(u,v)$ that contains the point $\left(\frac{x-1}{y},1\right)$ and also has slope less than that of $(x,y)$ will necessarily contain $(a,b)$, contradicting the assumption that $(x,y)$ wins at $(a,b)$.}
    \label{fig:top-edge-winners-proof}
\end{figure}

Lemma~\ref{lem:A_winners_are_Omega_winners} demonstrates that the problem of finding winning vectors on all of $\Omega$ reduces to finding winning vectors along the top edge. In light of the results of \cite{KSW}, we know this will be a finite set. Once we have found the winning vectors on $A$, we can determine the winning vector at every point of $\Omega$ by plotting the candidacy strips for each winner, which will necessarily cover $\Omega$. For any point that is contained in multiple strips, the winner at that point is given by the vector of least slope. This partitions $\Omega$ into a finite set of polygonal regions.

It turns out that we can refine the result of Lemma~\ref{lem:A_winners_are_Omega_winners} even further. 
Given any holonomy vector, we can consider the set of all points in $A$ where that vector wins. Our next result, Lemma~\ref{lem:W_is_interval}, tells us that this set is an interval of a specific form. 

We begin with two more definitions:
\begin{defn}
    For $(x,y) \in \Lambda$, define its \emph{candidacy interval} to be the set
    \begin{align*}
        I(x,y) &= S_{\Omega}(x,y) \, \cap \, \{b = 1\}. 
    \end{align*}
    Similarly to $A$, instead of a line segment in $\bbR^2$, we usually treat $I(x,y)$ as an interval in $\bbR$:
    \begin{align*}
        I(x,y) &= \left[ \frac{x-1}{y}, \frac{x}{y} \right).
    \end{align*}
    \label{def:candidacy-interval}
\end{defn}

\begin{defn}
    For $(x,y)\in \Lambda$, let $W(x,y)$ denote the subset of $A$ where $(x,y)$ is a winning vector. Like $A$ and $I(x,y)$, we usually treat $W(x,y)$ as a subset of $\bbR$.
    \label{def:winning-interval}
\end{defn}

Note that $W(x,y) \subseteq I(x,y)$ and $W(x,y) \subseteq A$. From the definition, it would seem that $W(x,y)$ could possibly be disconnected, but the next lemma proves that this is not the case. 

\begin{lem} \label{lem:W_is_interval}
    For all $(x,y) \in \Lambda$, either $W(x,y)$ is empty or $W(x,y) = \left[\frac{x-1}{y}, s\right)$ for some $s$ with $\frac{x-1}{y}<s \leq \min\left\{\frac{x}{y},\frac{x_0-1}{y_0}+n\alpha\right\}$.
\end{lem}

\begin{proof}
    Let $(x,y) \in \Lambda$. We begin by proving the following statement:
    \begin{equation}\label{eq:W_contains_interval}
        a \in W(x,y) \enspace \implies \enspace \left[ \frac{x-1}{y}, a \right] \cap A  \subseteq   W(x,y).
    \end{equation}
    Suppose $a \in W(x,y)$, and let $a' \in \left[ \frac{x-1}{y}, a \right] \cap A$. We want to show $a' \in W(x,y)$. 

    First, we know $a \in I(x,y) = \left[ \frac{x-1}{y}, \frac{x}{y} \right)$ because $W(x,y) \subseteq I(x,y)$. In particular, $a < \frac{x}{y}$. That means $I(x,y)$ contains $a'$, since 
    $$\frac{x-1}{y} \leq a' \leq a < \frac{x}{y}.$$
    Then since $a' \in A$, $(x,y)$ is a candidate winner at $a'$.
    
    Next, let $(u,v)$ be the winning vector at $a'$. We know $(x,y)$ is a candidate winner at $a'$ but does not beat $(u,v)$ at $a'$, so we must have  $\frac{x}{y} \leq \frac{u}{v}$. Also, we have $a' \in I(u,v) = \left[ \frac{u-1}{v}, \frac{u}{v} \right)$ because $W(u,v) \subseteq I(u,v)$. Then $I(u,v)$ contains $a$:
    $$\frac{u-1}{v} \leq a' \leq a < \frac{x}{y} \leq \frac{u}{v}.$$
    Then since $a \in W(x,y) \subseteq A$, $(u,v)$ is a candidate winner at $a$.

    We've found that $(u,v)$ is a candidate winner at $a$ and that $(x,y)$ is the winner at $a$. Thus we must have $\frac{u}{v} \leq \frac{x}{y}$. Combined with $\frac{x}{y} \leq \frac{u}{v}$, we conclude $(x,y)$ and $(u,v)$ have equal slopes. Recall that when multiple candidates have the least slope, the shortest candidate is the winner. The fact that $(u,v)$ is a candidate at $a$ while $(x,y)$ wins at $a$ implies $(u,v)$ is not shorter than $(x,y)$---otherwise, $(u,v)$ would be the winner at $a$ instead of $(x,y)$. Meanwhile, $(x,y)$ is a candidate at $a'$ while $(u,v)$ wins at $a'$, meaning $(x,y)$ is not shorter than $(u,v)$. So $(x,y)$ and $(u,v)$ have the same length. Their slopes are also equal, so $(x,y) = (u,v)$. Therefore, $(x,y)$ is the winning vector at $a'$, and $a' \in W(x,y)$, as desired. 

    Let $W(x,y)$ be nonempty. 
    By \ref{eq:W_contains_interval}, 
    we can write
    \begin{align*}
        W(x,y) &= \bigcup_{a \in W(x,y)} a \\
        &= \bigcup_{a \in W(x,y)} \left[ \frac{x-1}{y}, a \right] \cap A \\
        &= \left( \bigcup_{a \in W(x,y)} \left[ \frac{x-1}{y}, a \right] \right) \cap A.
    \end{align*}
    The union $\bigcup_{a \in W(x,y)} \left[ \frac{x-1}{y}, a \right]$ is an interval of the form $\left[\frac{x-1}{y}, s \right]$ or $\left[\frac{x-1}{y}, s \right)$, where $s=\sup W(x,y)$. Since $W(x,y)\subseteq A$, we have $s= \sup W(x,y)\leq \sup A = \frac{x_0-1}{y_0}+n\alpha$. Similarly, we have that $W(x,y)\subseteq I(x,y)$, which implies $s\leq \frac{x}{y}$.

    Since there is a unique winning vector at every point in $A$, the winning vectors partition $A$ into a finite set of intervals. Moreover, we already know from Lemma~\ref{lem:A_winners_are_Omega_winners} that the left endpoint of each interval, $\frac{x_-1}{y}$, is contained in $A$ for each winning vector $(x,y)$. Since the right endpoint of $A$ itself is open and each winning interval contains its left endpoint, this implies that the right endpoint of each of the winning intervals must be open. Therefore, we have $W(x,y)=\left[\frac{x-1}{y}, s \right)\cap A$.

    Finally, we observe that intersecting with $A$ is redundant. Since $\frac{x_0-1}{y_0}\leq \frac{x-1}{y}$ and $s\leq \frac{x_0-1}{y_0}+n\alpha$, we have 
    \begin{align*}
        W(x,y) = \left[\frac{x-1}{y}, s \right)\cap A
        = \left[\frac{x-1}{y}, s \right) \cap \left[\frac{x_0-1}{y_0}, \frac{x_0-1}{y_0}+n\alpha \right)
        = \left[\frac{x-1}{y}, s \right)
    \end{align*}
    which yields the result.

\end{proof}

This lemma tells us that $A$ decomposes into a finite number of half-open intervals, with the left endpoint of each interval determined formulaically by the vector that wins there.

\subsection{Left Winning Vectors}
\label{sec:left-winners}

For the algorithm that we present in the next subsection, it will be convenient to work with slightly modified definitions of several of the objects introduced above. Each modification amounts to only changing which edge or endpoint we include in a given set.
\begin{defn}\label{def:left-defs}
Let 
\[
\Omega^L = \left\{ (a,b) \in \bbR^2 \mmid 0 < b \leq 1, \; \frac{x_0}{y_0} b - \frac{1}{y_0} < a \leq \; \right(\frac{x_0}{y_0} + n\alpha\left) b -\frac{1}{y_0} \right\}
\]
and let $A^L$ be its top edge, identified with the interval
\[
A^L = \left(\frac{x_0-1}{y_0}, \frac{x_0-1}{y_0}+n\alpha \right].
\]
We say a vector $(x,y)\in\Lambda$ is a \emph{candidate left winning vector} at $a \in \Omega^L$ if 
it satisfies $y > 0$ and $0 \leq bx - ay < 1$.
\footnote{We designate these definitions as ``left" (somewhat arbitrarily) because we include equality in the left inequality of $0\leq bx-ay< 1$ rather than the right inequality, as in Definition~\ref{def:ourcandidacystrips}} We say it is a \emph{left winning vector} or \emph{left winner} if it is the candidate left winning vector of least slope. If there are multiple vectors of least slope, we take the shortest as the winner. 
    
Additionally, for $(x,y)\in\Lambda$ and $(a,b)\in\Omega^L$, we define the candidacy strips
\begin{align*}
S^L_{\Omega}(x,y) &= \{ (a,b)\in\bbR^2 \; \vert \; 0 \leq bx-ay < 1,\; b>0 \}\\
S^L_{\Lambda}(a,b) &= \{ (x,y)\in\bbR^2 \; \vert \; 0 \leq bx-ay < 1,\; y>0 \}.
\end{align*}
We define $I^L(x,y)$ to be $S_\Omega^L(x,y)\cap\{b=1\}$, identified with the interval
\[
I^L(x,y) = \left(\frac{x-1}{y},\frac{x}{y}\right],
\]
and we define $W^L(x,y)$ to be the subset of $A^L$ on which $(x,y)$ is the left winning vector. 
\end{defn}
We briefly remark that these are purely algebraic definitions---in fact, a ``left winning vector" that wins at the right edge of its candidacy strip, where $bx-ay=0$, \emph{cannot} be a true winning vector at the point $(a,b)
\in\Omega$. This is because a winning vector at that point corresponds to the first holonomy vector $v\in\Lambda$ such that $M_{b,-a}v$ will hit the $x$-axis between $0$ and $1$ when flowed forward by the horocycle flow. However, when $bx-ay=0$, we have that $M_{b,-a}\cdot (x,y) = (0,b^{-1}y)$, which is fixed under the horocycle flow, so it will never hit the $x$-axis. 

Regardless, we see that the respective definitions for the two kinds of winners and their associated sets agree almost everywhere, and the next lemma shows that finding the left winning vectors on $A^L$ is in fact equivalent to finding the winning vectors on $A$.

\begin{lem}\label{lem:left-winners-are-right-winners}
    Let $(x,y)\in\Lambda$. Then $W(x,y) = \left[\frac{x-1}{y},s\right)$ on $A$ if and only if $W^L(x,y) = \left(\frac{x-1}{y},s\right]$ on $A^L$ for some $s\in A^L\cap I^L(x,y)$.
\end{lem}

\begin{proof}
    Recall that $\cup_{v\in\Lambda} I(v)$ covers $A$. This implies that $\cup_{v\in\Lambda}I^L(v)$ covers $A^L$, so every point on $A^L$ has at least one candidate left winning vector at that point. Then from the definition of left winners and the discreteness of $\Lambda$, it follows that there is a unique left winning vector at every point on $A^L$.

    Suppose $W(x,y) \neq \emptyset$. By Lemma~\ref{lem:W_is_interval}, we know that $W(x,y) = \left[\frac{x-1}{y},s\right)$ for some $s\in A^L\cap I^L(x,y)$. 
    
    First, we will show $s\in W^L(x,y)$. Let $(u,v)$ be the left winner at $s$. Then $s\in I^L(u,v)$, but we also know $s\in I^L(x,y)$. This implies $\frac{x}{y}\leq\frac{u}{v}$, since otherwise $(x,y)$ would beat $(u,v)$ as the left winner at $s$. It also implies that $\frac{u-1}{v}<s$ and $\frac{x-1}{y}<s$, so there exists some $a'$ with $\max\left\{\frac{u-1}{v}, \frac{x-1}{y}\right\} <a' <s$. Then $a'\in W(x,y)\cap I(u,v)$, which implies $\frac{u}{v}\leq \frac{x}{y}$. Since $(x,y)$ is the winning vector at $a'$ and $(u,v)$ is the left winning vector at $s$, we also know that neither vector can be shorter than the other. Therefore, $(u,v)=(x,y)$, so $s\in W^L(x,y)$.

    Now consider $\frac{x-1}{y}<a<s$ and let $(u,v)$ be the left winning vector at $a$. We observe that $a\in I^L(x,y)\cap W^L(u,v)$, so $\frac{x}{y}\leq\frac{u}{v}$. On the other hand, $a\in W(x,y)\cap I(u,v)$, so $\frac{u}{v}\leq \frac{x}{y}$. As neither vector can be shorter than the other, $(u,v)=(x,y)$, so $a\in W^L(x,y)$.

    Next, consider $s<a\leq \frac{x_0-1}{y_0}+n\alpha$, if such an $a$ exists (if $s=\frac{x_0-1}{y_0}+n\alpha$, then we have already handled this case above). Then there exists $a'$ such that $s<a'<a$. Suppose for the sake of contradiction that $(x,y)$ was the left winner at $a$. Observe that since $a\in I^L(x,y)$, we know $\frac{x-1}{y} <s<a'<a\leq \frac{x}{y}$, so $a'\in I(x,y)\cap I^L(x,y)$. Moreover, we know that $(x,y)$ does not win at $a'$, but that $a'\in A$, so there must exist another vector, $(u,v)\neq(x,y)$, that wins at $a'$. Hence $a'\in W(u,v)\cap I(x,y)$, which implies $\frac{x}{y}\leq \frac{u}{v}$. But since $\frac{u-1}{v}<a'<a\leq\frac{x}{y}\leq\frac{u}{v}$, we see that $a\in I^L(u,v)$. 
    But since $(x,y)$ was assumed to be the left winner at $a$, this implies $\frac{u}{v}\leq \frac{x}{y}$. Since $(x,y)$ is the left winner over $(u,v)$ at $a$ and $(u,v)$ is the winner over $(x,y)$ at $a'$, neither vector can be shorter than the other. This contradicts the assumption that $(u,v)\neq(x,y)$. We conclude that $a\notin W^L(x,y)$.
    
    Finally, if $a\leq \frac{x-1}{y}$ or $a>\frac{x_0-1}{y_0}+n\alpha$, then $a\notin I^L(x,y)\cap A^L$, so $a\notin W^L(x,y)$. 

    We have shown that $W(x,y) = \left[\frac{x-1}{y},s\right)\subseteq A$ implies $W^L(x,y) = \left(\frac{x-1}{y},s\right] \subseteq A^L$. For the reverse implication, note that we have already shown that the true winning vectors partition $A$ into a finite collection of intervals of the given form. The forward implication applied to each of these intervals completely exhausts the set $A^L$, so there can be no left winning vectors on $A^L$ that do not satisfy the theorem. 
\end{proof}

\subsection{Algorithm to Find Left Winning Vectors}\label{sec:algorithm}

In Lemmas~\ref{lem:A_winners_are_Omega_winners} and \ref{lem:W_is_interval}, we showed that winning vectors on $\Omega$ win on half-open intervals of $A$, and we gave a formula for the left endpoint of each interval. In Lemma~\ref{lem:left-winners-are-right-winners}, we showed that winning vectors on $A$ are in one-to-one correspondence with so-called ``left winning vectors," on $A^L$, cf. Definition~\ref{def:left-defs}. The left winners win on half-open intervals which are open on the left and closed on the right.\footnote{This is why we work with left winners on $A^L$ instead of winners on $A$. Since the winning intervals for left winners are open on the left, at each step of the algorithm we can search for the left winner at the precise point that is the previous left winning interval's infimum, rather than at some undefined point to the left of the previous winning interval's minimum, as would be the case for true winners.}

The idea behind our algorithm is as follows: Start by finding the left winning vector at the rightmost point of $A^L$. By Lemma~\ref{lem:W_is_interval}, that vector wins along an interval extending leftwards along $A^L$, and we know the precise point where that interval ends. We then repeat the process, finding the left winner at that endpoint and following its winning interval leftwards along $A^L$. Continuing in this way, we work right-to-left until we have determined the winning vectors along all of $A^L$. This reduces the problem of finding all winning vectors on $\Omega$ to finding the left winning vector at an unknown but finite number of points along $A^L$.

\begin{defn}[Algorithm to find winning vectors on $\Omega$] 
\label{def:algorithm}
    \phantom{filler}
   \begin{enumerate}
        \item Initialize the variable $i = 1$, and set $a_i = a_1 = \frac{x_0 - 1}{y_0}+n\alpha$.
        \item Find the left winning vector at $a_i$ (see Section~\ref{sec:winner_at_point} for a technique to do this). Call this vector $(u_i,v_i)$.
        \item Set $a_{i+1} = \frac{u_i - 1}{v_i}$. If $(u_i,v_i) = (x_0, y_0)$, end the algorithm. Otherwise, increment $i$ up by 1, and then return to step (2).
    \end{enumerate}
\end{defn}

\begin{thm} \label{thm:alg_finds_winners}
    Assuming one can find the left winning vectors at the points $a_i$, the algorithm in Definition~\ref{def:algorithm} terminates, and $\{ (u_i, v_i) \}$ is the set of winning vectors on $\Omega$.
\end{thm}

\begin{proof}

First, we show that $W^L(u_i,v_i) = (a_{i+1}, a_i]$ for all $i$. We proceed by induction.

    \textit{Base case:} When $i = 1$, we know $(u_1, v_1)$ is the left winning vector at $a_1$, so $a_1 \in W^L(u_1,v_1)$. Since $\sup A^L=a_1$ and $W^L(u_1,v_1) \subseteq A$, we find that $\sup W^L(u_1,v_1)=a_1$ as well. Then Lemma~\ref{lem:W_is_interval} tells us $W^L(u_1,v_1) = \left( \frac{u_1-1}{v_1}, a_1 \right] = (a_2, a_1]$, completing the base case.

    \textit{Inductive step:} Suppose $W^L(u_{i-1},v_{i-1}) = (a_i, a_{i-1}]$. The left winning vector at $a_i$ is $(u_i, v_i)$, so $a_i \in W^L(u_i, v_i)$. Note that $(u_i, v_i) \neq (u_{i-1}, v_{i-1})$ since $(u_i, v_i)$ wins at a point that $(u_{i-1}, v_{i-1})$ does not win at. Then $W^L(u_i, v_i)$ must be disjoint from $W^L(u_{i-1}, v_{i-1}) = (a_i, a_{i-1}]$. Also, Lemma~\ref{lem:W_is_interval} tells us $W^L(u_i, v_i)$ is an interval. So $W^L(u_i, v_i)$ is an interval containing $a_i$ and disjoint from $(a_i, a_{i-1}]$, meaning $\sup W^L(u_i, v_i)=a_i$. Finally, Lemma~\ref{lem:W_is_interval} tells us $W^L(u_i, v_i) = \left( \frac{u_i - 1}{v_i}, a_i\right] = (a_{i+1}, a_i]$, completing the inductive step.

We know that there is a unique left winning vector at each point of $A^L$. Moreover, Lemmas~\ref{lem:A_winners_are_Omega_winners} and \ref{lem:left-winners-are-right-winners} tell us that these are the same as the winning vectors on $\Omega$, which by \cite{KSW}, we know to be a finite set. Since each step of the algorithm produces a distinct left winning vector on $A^L$, the algorithm must therefore terminate. Additionally, from \cite{KSW}, we know that the vector $(x_0,y_0)$ wins in a neighborhood of the left endpoint of $A^L$, so we know that the algorithm ends once we reach a point where $(x_0,y_0)$ wins. 

Letting $(u_n,v_n) =(x_0,y_0)$, we know that
    $$\bigcup_{i=1}^n \ W^L(u_i,v_i) 
    = 
    \bigcup_{i=1}^n \ (a_{i+1}, a_i] 
     = 
    (a_n, a_1] 
     = 
    \left( \frac{x_0 - 1}{y_0}, \frac{x_0-1}{y_0} + n\alpha \right]
     = 
    A^L.$$
Thus $\{(u_i,v_i)\}_{i=1}^n$ exhausts the set of winning vectors on $A^L$ (hence $\Omega$).
    
\end{proof}

\subsection{Finding the Winning Vector at a Point}\label{sec:winner_at_point}

To successfully apply the algorithm in Definition~\ref{def:algorithm}, we need to be able to find the left winning vector at each endpoint. 

Suppose we knew the winning vector lay in a specified bounded region. That region would contain finitely many holonomy vectors since $\Lambda$ is discrete. It would then take finitely many calculations to find all those holonomy vectors' coordinates, and to test them one by one to identify the winner. In Lemma~\ref{lem:region_with_winner}, we observe that for $(a,b)$ in either $\Omega$ or $\Omega^L$, knowing a single candidate winning vector at $(a,b)$ is often enough to generate such a bounded region.

\begin{lem} \label{lem:region_with_winner}
    Let $(u,v)\in\Lambda$ be a candidate (left) winner at $(a,b)\in\Omega^{(L)}$. Then the (left) winner at $(a,b)$ lies in the region
    \begin{align}
        R=S_{\Lambda}^{(L)}(a,b) \; \cap \; \left\{ (x,y) \in \bbR^2 \enspace \middle| \enspace vx - uy \geq 0 \right\}.
        \label{eq:winner-region}
    \end{align}
    This region is bounded if and only if $bu - av \neq 0$; in other words, as long as $(u,v)$ and $(a,b)$ have different slopes.
\end{lem}

\begin{proof}
    Let $(u',v')$ be the (left) winning vector at $(a,b)$. In particular, $(u',v')$ is a candidate at $(a,b)$, so $(u',v') \in S^{(L)}_\Lambda(a,b)$ by definition. By assumption, we know $(u,v)$ is also a candidate at $(a,b)$. By the definition of the (left) winning vector at $(a,b)$, $(u',v')$ must have slope no greater than that of $(u,v)$, i.e.\ $\frac{u}{v}\leq\frac{u'}{v'}$. Since $(u,v),(u',v')\in S^{(L)}(a,b)$ implies $v,v'>0$, this means $vu' - uv' \geq 0$, so $(u',v') \in \left\{ (x,y) \in \bbR^2 \enspace \middle| \enspace vx - uy \geq 0 \right\}$. We conclude that $(u',v')\in R$. 
    
    Figure~\ref{fig:bounded_region} demonstrates when this region is bounded. The set $S^{(L)}_{\Lambda}(a,b)$ contains points in the upper half-plane between two non-horizontal parallel lines, $bx - ay = 0$ and $bx - ay = 1$. The line $vx - uy = 0$ intersects both those parallel lines if and only if $bu - av \neq 0$. When that is the case, intersecting the half-plane under $vx - uy = 0$ with $S_{\Lambda}(a,b)$ produces a bounded region. So $R$ 
    is bounded if and only if $bx - ay \neq 0$.
\end{proof}

    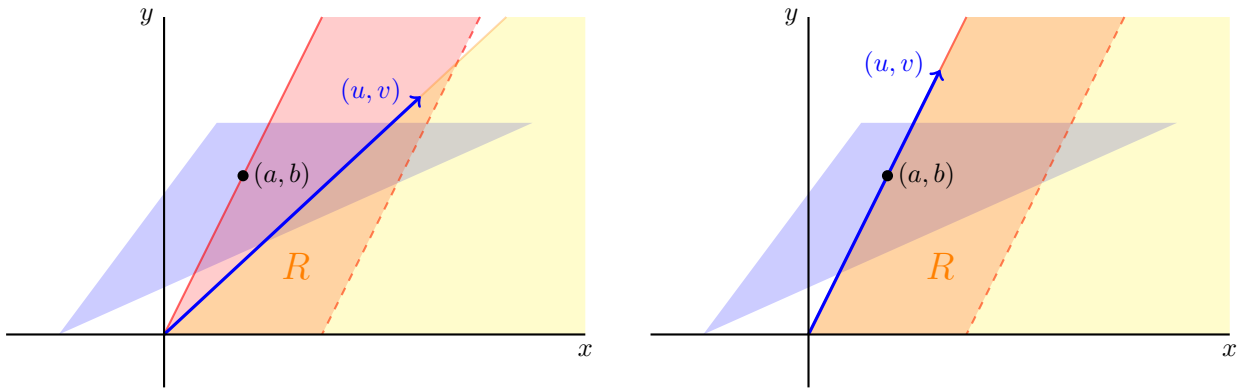
\begin{figure}
    \begin{tikzpicture}[scale=0.7]
        \fill[fill=blue, opacity=0.2] (-2,0)--(1,4)--(7,4);
        \fill[fill=red, opacity=0.2] (0,0)--(3,6)--(6,6)--(3,0);
        \draw[thick,red,opacity=0.6] (0,0)--(3,6);
        \draw[thick,dashed,red,opacity=0.6] (6,6)--(3,0);
        \fill[fill=yellow, opacity=0.2] (6.5,6)--(0,0)--(8,0)--(8,6);
        \draw[thick,orange,opacity=0.3] (6.5,6)--(0,0);
        \draw[blue,->,very thick] (0,0)--(4.875,4.5);
        \node[blue,left] at (4.7,4.6) {\small $(u,v)$};
        \fill[black] (1.5,3) circle (3pt);
        \node[right] at (1.5,3) {\small $(a,b)$};
        \node[below] at (8,0) {\small $x$};
        \node[left] at (0,6) {\small $y$};
        \draw[thick] (-3,0)--(8,0);
        \draw[thick] (0,-1)--(0,6);
        \node[orange] at (2.5,1.3) {\Large $R$};
    \end{tikzpicture}
    \hspace{1em}
    \begin{tikzpicture}[scale=0.7]
        \fill[fill=blue, opacity=0.2] (-2,0)--(1,4)--(7,4);
        \fill[fill=red, opacity=0.2] (0,0)--(3,6)--(6,6)--(3,0);
        \draw[thick,red,opacity=0.6] (0,0)--(3,6);
        \draw[thick,dashed,red,opacity=0.6] (6,6)--(3,0);
        \fill[fill=yellow, opacity=0.2] (3,6)--(0,0)--(8,0)--(8,6);
        \draw[blue,->,very thick] (0,0)--(2.5,5);
        \node[blue,left] at (2.4,5.1) {\small $(u,v)$};
        \fill[black] (1.5,3) circle (3pt);
        \node[right] at (1.5,3) {\small $(a,b)$};
        \node[below] at (8,0) {\small $x$};
        \node[left] at (0,6) {\small $y$};
        \draw[thick] (-3,0)--(8,0);
        \draw[thick] (0,-1)--(0,6);
        \node[orange] at (2.5,1.3) {\Large $R$};
    \end{tikzpicture}
    \caption{The red strip is $S^L_\Lambda(a,b)$ and the yellow sector between the $x$-axis and $(u,v)$ contains all holonomy vectors with slope shallower than that of $(u,v)$. If $(u,v)$ is in the interior of the strip, i.e.\ if $bu - av \neq 0$, then their orange intersection, $R$, is bounded (left). When $bu - av = 0$, this intersection is not bounded (right). Note that although we have overlaid $\Omega$ and $(a,b)$ in the figure, we are searching for holonomy vectors that might beat $(u,v)$ in the orange region in the plane of $\Lambda$.}
    \label{fig:bounded_region}
\end{figure}

Notice that $(u,v)$ lies on the left edge of the region $R$ defined in \eqref{eq:winner-region},
so if there is any holonomy vector in $R$ of slope different from $(u,v)$, then it \emph{will} beat $(u,v)$ at $(a,b)$. Therefore, if $(u,v)$ is the smallest holonomy vector of its slope, then proving that $R$ is empty of any holonomy vectors of different slope is equivalent to proving that $(u,v)$ wins at $(a,b)$.

If the set in \eqref{eq:winner-region} is not bounded, then it is an infinite strip and $(u,v)$ lies on the edge of this strip. In this case, it can be harder to prove that the set does not contain any holonomy vectors that might win against $(u,v)$. One approach is to shear the translation surface, which in turn shears the unbounded strip, and choose the shear so that the sheared strip is vertical. This vertical strip has a known width. Then, the question of whether the (sheared) strip contains any holonomy vectors in its interior boils down to whether there exist any (sheared) holonomy vectors whose $x$-components are strictly less than that width. For an example of this approach in practice, see our proof of the slope gap distribution of the double heptagon translation surface in the next section.

\section{The Double Heptagon}\label{sec:dbl_heptagon}

In this section, we apply the algorithm from Section~\ref{sec:algorithm} to the double heptagon translation surface to calculate its renormalized slope gap distribution.

\subsection{The Double Heptagon and Staircase Representation}\label{sec:7gon_setup} 
Let $(X,\omega)$ denote the double heptagon translation surface formed by two regular heptagons with unit side lengths, where same-color sides are identified as shown in Figure~\ref{fig:double-heptagon}. This is a genus $3$ translation surface with a single cone point of total angle $10\pi$.

\begin{figure}[h]
    \begin{tikzpicture}[scale =1.3]
    \draw [black, ultra thick] (0,0) --(1,0);
    \draw [orange!80, ultra thick] (1,0) --(1.623,0.782);
    \draw [orange!80, ultra thick] (0,0) --(-0.62,-0.765);
    \draw [green!60, ultra thick] (-0.62,-0.765) --(-0.4,-1.757);
    \draw [green!60, ultra thick] (1.623,0.782) -- (1.401,1.757);
    \draw [red!80, ultra thick] (1.401,1.757) -- (0.5,2.191);
    \draw [red!80, ultra thick] (-0.4,-1.757) --(0.5,-2.191);
    \draw [blue!65, ultra thick] (0.5,-2.191) --(1.4,-1.757);
    \draw [blue!65, ultra thick] (0.5,2.191) --(-.401,1.757);
    \draw [yellow, ultra thick] (-.401,1.757) --(-.623,0.782);
    \draw [yellow, ultra thick] (1.623,-.782) --(1.4,-1.757);
    \draw [violet!80, ultra thick] (1.623,-.782) --(1,0);
    \draw [violet!80, ultra thick] (-.623,.782) --(0,0);
    \end{tikzpicture}
    \caption{The double heptagon surface. Sides of the same color are identified, and all vertices are idenitified, representing a single cone point of angle $10\pi$.}
    \label{fig:double-heptagon}
\end{figure}
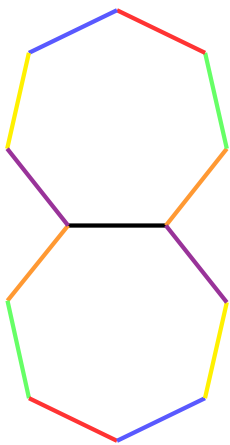

In \cite{Veech89}, Veech showed that the Veech group of this surface is given by $\slxo=\langle Q,P\rangle$, where
\[
Q = \begin{pmatrix} \cos\frac{\pi}{7}&-\sin\frac{\pi}{7}\\\sin\frac{\pi}{7}&\cos\frac{\pi}{7}\end{pmatrix},
\qquad\qquad
P=\begin{pmatrix} 1&2\cot\frac{\pi}{7}\\0&1\end{pmatrix}.
\]
This is isomorphic to the $(2,7,\infty)$ triangle group and has one cusp corresponding to the parabolic generator $P$, with $C=\Id$, $\alpha = 2\cot(\pi/7)$, and $(x_0,y_0) = (\cos(\pi/7),\sin(\pi/7))$. We also see that $-\Id\in \slxo$, so the Poincar\'e section for $(X,\omega)$ has just one component, given by 
\[
\Omega = \left\{ (a,b) \in \bbR^2 \mmid 0 < b \leq 1, \; \cot\left(\frac{\pi}{7}\right) b - \csc\left(\frac{\pi}{7}\right) \leq a < \; 3\cot\left(\frac{\pi}{7}\right) b - \csc\left(\frac{\pi}{7}\right) \right\}.
\]

In order to use the method from Section~\ref{sec:winner_at_point}, it will be helpful to know an easily computable discrete set $L$ containing all of the possible winning vectors on $\Omega$. To get such a superset, we note that $(X,\omega)$ admits a ``staircase'' presentation which is cut-translate-paste equivalent to the original double heptagon; see Figure~\ref{fig:staircase_surface}.
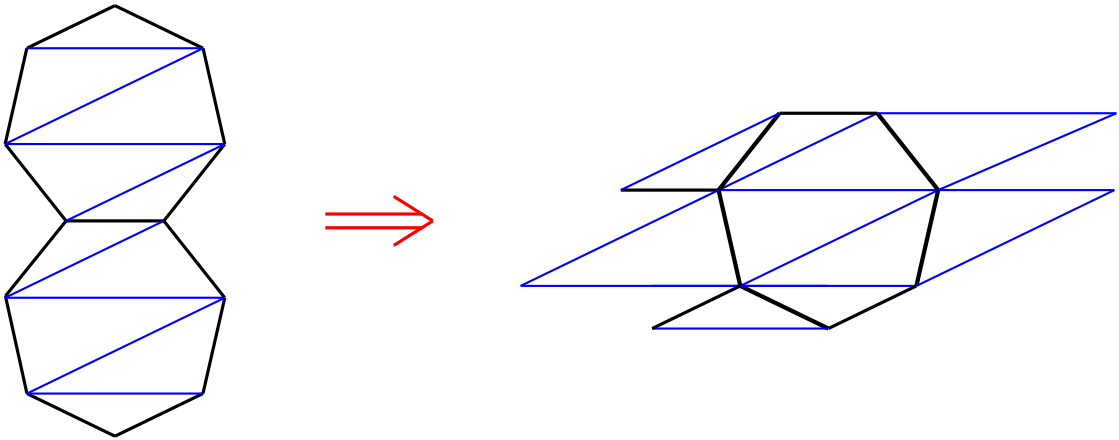
\begin{figure}
\centering
\begin{subfigure}[c]{0.25\textwidth}
    \begin{tikzpicture}[scale =1.3]
    \draw [black,very thick] (0,0) --(1,0);
    \draw [black, very thick] (1,0) --(1.623,0.782);
    \draw [black, very thick] (0,0) --(-0.62,-0.765);
    \draw [black, very thick] (-0.62,-0.765) --(-0.4,-1.757);
    \draw [black, very thick] (1.623,0.782) -- (1.401,1.757);
    \draw [black, very thick] (1.401,1.757) -- (0.5,2.191);
    \draw [black, very thick] (-0.4,-1.757) --(0.5,-2.191);
    \draw [black, very thick] (0.5,-2.191) --(1.4,-1.757);
    \draw [black, very thick] (0.5,2.191) --(-.401,1.757);
    \draw [black, very thick] (-.401,1.757) --(-.623,0.782);
    \draw [black, very thick] (1.623,-.782) --(1.4,-1.757);
    \draw [black, very thick] (1.623,-.782) --(1,0);
    \draw [black, very thick] (-.623,.782) --(0,0);
    \draw [blue,  thick] (-.401,1.757) --(1.401,1.757);
    \draw [blue,  thick] (-.623,.782) --(1.401,1.757);
    \draw [blue,  thick] (-.623,.782) --(1.623,.782);
    \draw [blue,  thick] (0,0) --(1.623,.782);
    \draw [blue,  thick] (1,0) --(-0.623,-.782);
    \draw [blue,  thick] (1.623,-.782) --(-0.623,-.782);
    \draw [blue,  thick] (1.623,-.782) --(-0.4,-1.757);
    \draw [blue,  thick] (1.4,-1.757) --(-0.4,-1.757);
    \end{tikzpicture}
\end{subfigure}
\begin{subfigure}[c]{0.15\textwidth}
\begin{tikzpicture}[scale =1.3]
    \draw [red, very thick] (3,-0.07) --(4,-0.07);
    \draw [red, very thick] (3,0.07) --(4,0.07);
    \draw [red, very thick] (4.1,0) --(3.7,0.25);
    \draw [red, very thick] (4.1,0) --(3.7,-0.25);
\end{tikzpicture}
\end{subfigure}
\begin{subfigure}[c]{0.45\textwidth}
    \begin{tikzpicture}[scale =1.3]
    \draw [blue,  thick] (8,0) --(6.377,-.782);
    \draw [blue,  thick] (8.623,-.782) --(6.377,-.782);
    \draw [blue,  thick] (8.623,-.782) --(6.6,-1.757);
    \draw [blue,  thick] (8.4,-1.757) --(6.6,-1.757);
    \draw [black,  very thick] (7.5,-2.191) --(8.4,-1.757);
    \draw [black, ultra thick] (7.5,-2.191) --(6.599,-1.757);
    \draw [black, ultra thick] (6.599,-1.757) --(6.377,-0.782);
    \draw [black, ultra thick] (8.623,-.782) --(8.4,-1.757);
    \draw [black, ultra thick] (8.623,-.782) --(8,0);
    \draw [black, ultra thick] (6.377,-.782) --(7,0);
    \draw [black, very thick] (8,0) --(7,0);
    \draw [black, very thick] (6.377,-.782) --(5.377,-.782);
    \draw [blue,  thick] (7,0) --(5.377,-.782);
    \draw [blue,  thick] (4.353,-1.757) --(6.377,-.782);
    \draw [blue,  thick] (6.599,-1.757) --(4.353,-1.757);
    \draw [blue,  thick] (7.5,-1.757) --(5.698062264,-1.757);
    \draw [blue,  thick] (7.5,-2.190883739) --(5.698062264,-2.190883739);
    \draw [black, very thick] (6.599,-1.757) --(5.698062264,-2.190883739);
    \draw [blue,  thick] (8,0) --(10.445041868,0);
    \draw [blue,  thick] (8.623,-.782) --(10.445041868,0);
    \draw [blue,  thick] (10.42493774,-.782) --(8.623,-.782);
    \draw [blue,  thick] (10.42493774,-.782) --(8.4,-1.757);
    \end{tikzpicture}
\end{subfigure}
    \caption{Staircase presentation of the double heptagon translation surface.}\label{fig:staircase_surface}
\end{figure}

Viewing the staircase as three stacked parallelograms, $(x_0,y_0)$ is the left edge of the bottom parallelogram. Each non-horizontal side is parallel to $(x_0,y_0)$. 

Rather than working with the slanted staircase, we will work with a rectilinear one. The horizontal shear
\[
M = \begin{pmatrix}
1 & -\cot\left(\frac{\pi}{7}\right) \\
0 & \csc\left(\frac{\pi}{7}\right) 
\end{pmatrix}
\]
brings $(x_0, y_0)$ to the vertical vector $(0,1)$. The segments parallel to $(x_0,y_0)$ become vertical, too, while horizontal segments remain horizontal. That means $M\cdot (X,\omega)$ is a rectilinear staircase, as shown in Figure~\ref{fig:rectilinear_surface}.
\begin{figure}
    \begin{tikzpicture}[scale =1.3]
    
    \coordinate (a) at (0,0);
    \coordinate (b) at (1.8019,0);
    \coordinate (c) at (1.8019,3.247);
    \coordinate (d) at (0,3.247);
    \coordinate (e) at (0,5.049);
    \coordinate (f) at (-3.247,5.049);
    \coordinate (g) at (-3.247,3.247);
    \coordinate (h) at (-2.247,3.247);
    \coordinate (i) at (-2.247,1);
    \coordinate (j) at (0,1);

    \coordinate (k) at (1.8019,1);
    \coordinate (l) at (-2.247,5.049);

    \draw[-] (a) to (b);
    \draw[-] (b) to (c);
    \draw[-] (c) to (d);
    \draw[-] (d) to (e);
    \draw[-] (e) to (f);
    \draw[-] (f) to (g);
    \draw[-] (g) to (h);
    \draw[-] (h) to (i);
    \draw[-] (i) to (j);
    \draw[-] (j) to (a);

    \draw[dashed] (j) to (k);
    \draw[dashed] (j) to (d);
    \draw[dashed] (d) to (h);
    \draw[dashed] (h) to (l);

    \node (l1) at (-2.8,2.92) {$l_1 = 1$};
    \node (l2) at (0.9,-0.4) {$l_2 = 2\cos(\pi/7)$};
    \node (l3) at (-1.25,0.6) {$l_3 = \frac{1}{2}\csc(\pi/14)$};

    \end{tikzpicture}
    \caption{The shear $M$ applied to the double heptagon translation surface. All horizontal and vertical edges have lengths $l_1, l_2,$ or $l_3$.}\label{fig:rectilinear_surface}
\end{figure}
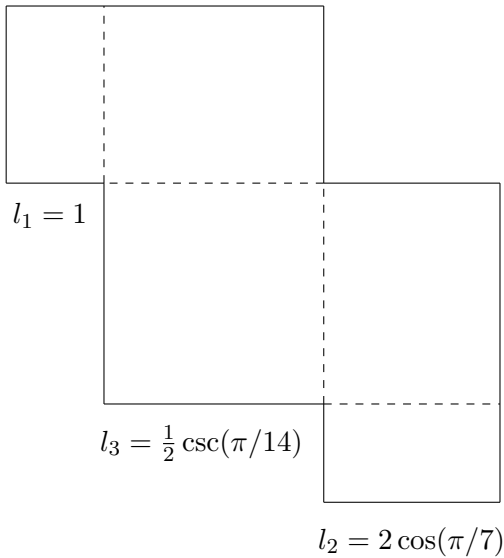

When we apply $M$ to the surface, we also apply it to its holonomy vectors; that is, the set of holonomy vectors on $M\cdot (X,\omega)$ is $M\cdot \Lambda$. We care about the subset of $\Lambda$ that could contain winning vectors, meaning the set of holonomy vectors with positive $y$-coordinate and positive slope shallower than that of $(x_0,y_0)$ (note that since $(x_0,y_0)$ is known to win in a neighborhood of the left endpoint of $A$, all other winning vectors must have smaller slope). Applying $M$ to this set gives us the following set:
\begin{defn}
     Let $\mathcal{L}$ be the set of holonomy vectors on $M\cdot (X,\omega)$ in the first quadrant of $\R^2$.
\end{defn}

While $\mathcal{L}$ is hard to describe explicitly, it admits a discrete superset with a simple explicit description. The horizontal and vertical distances between vertices of $M\cdot (X,\omega)$ are integer sums of three distinct lengths, $l_1 = 1, l_2 = 2\cos(\pi/7)$, and $l_3 = 
\frac{1}{2}\csc(\pi/7)$. This means a holonomy vector on $M\cdot (X,\omega)$ must have $x$- and $y$-components that are integer sums of $l_1,l_2,$ and $l_3$. Thus elements of $\mathcal{L}$ belong to the following set:
\begin{defn}
    Let $L = (\bbN \cdot l_1 + \bbN \cdot l_2 + \bbN \cdot l_3)^2$. 
\end{defn}

This is a discrete superset of $\mathcal{L}$ that be easily generated by a computer program for all vectors up to a certain length. We will use this to search through possible holonomy vectors that could win at some point on $\Omega$.

\subsection{Running the Algorithm}

We apply the algorithm described in Definition~\ref{def:algorithm} to obtain the following exhaustive list of left winning vectors on $A^L$. By the results of Section~\ref{sec:winners}, we know that these are all of the winning vectors on $\Omega$.

\begin{table}
    \centering
        \setlength{\tabcolsep}{5pt}
        \renewcommand{\arraystretch}{2.5}
        \begin{tabular}{ |c|c|c| }
            \hline
            Point in $A^L$ & Left winner & $M$ applied to left winner\\
            \hline
            $a_1=\frac{3\cos\left(\frac{\pi}{7}\right) - 1}{\sin\left(\frac{\pi}{7}\right)}$
            &
            $w_1 = \left(2 + 3\cos\left(\frac{2\pi}{7}\right), \sin\left(\frac{2\pi}{7}\right)\right)$
            &
            $(l_3,l_2)$
            \\ 

            $a_2=\frac{1 + 3\cos\left(\frac{2\pi}{7}\right)}{\sin\left(\frac{2\pi}{7}\right)}$
            &
            $w_2 = \left(4\cos\left(\frac{\pi}{7}\right) + 3\cos\left(\frac{3\pi}{7}\right), \sin\left(\frac{3\pi}{7}\right)\right)$
            &
            $(l_3,l_3)$
            \\ 
            
            $a_3=\frac{4\cos\left(\frac{\pi}{7}\right) + 3\cos\left(\frac{3\pi}{7}\right) - 1}{\sin\left(\frac{3\pi}{7}\right)}$
            &
            $w_3 = \left(4\cos\left(\frac{\pi}{7}\right) + \cos\left(\frac{3\pi}{7}\right), \sin\left(\frac{3\pi}{7}\right)\right)$
            &
            $(l_2,l_3)$
            \\ 
            
            $a_4=\frac{4\cos\left(\frac{\pi}{7}\right) + \cos\left(\frac{3\pi}{7}\right) - 1}{\sin\left(\frac{3\pi}{7}\right)}$
            &
            $w_4 = \left(2 + \cos\left(\frac{2\pi}{7}\right), \sin\left(\frac{2\pi}{7}\right)\right)$
            &
            $(l_1,l_2)$
            \\

            $a_5=\frac{1 + \cos\left(\frac{2\pi}{7}\right)}{\sin\left(\frac{2\pi}{7}\right)}$
            &
            $w_5 = \left(\cos\left(\frac{\pi}{7}\right), \sin\left(\frac{\pi}{7}\right)\right)$
            &
            $(0,l_1)$
            \\
            \hline
        \end{tabular}
    \caption{The left winners on $A^L$, along with their corresponding vectors in $M\cdot \Lambda$. Note that $w_5=(x_0,y_0)$, which tells us that we can end the algorithm.}
    \label{fig:winner-table}
\end{table}

\begin{lem} \label{lem:7-gon_winners_at_points}
    The vectors $w_1, \ldots, w_5$ given in Table~\ref{fig:winner-table} are the left winners at the points $a_1,\dots, a_5\in A^L$, respectively, which are the right endpoints of their winning intervals.
\end{lem}
\begin{proof}
    First, a quick calculation shows that each of our proposed winners is a candidate left winning vector at their corresponding points $a_i\in A$. That means we can use Lemma~\ref{lem:region_with_winner} to prove these five vectors win. This amounts to showing that the five regions $R_1, \ldots, R_5$ as defined \eqref{eq:winner-region} each contain no holonomy vectors in their interiors. Equivalently, they should contain no such elements of $M^{-1}\cdot \mathcal{L}$. Note that the proposed winners themselves are always in these regions, but for convenience, we will still say they are ``empty'' if they contain no holonomy vectors of different slope, i.e.\ no vectors in their interiors.
    
    For $R_1$ through $R_4$, our strategy is as follows: Using a Python program,\footnote{Code for this program can be found in Appendix~\ref{sec:Python-L}.} 
    we find all vectors in $(M\cdot R_i) \cap L$. Then, we manually check whether each of those vectors is actually a holonomy vector, i.e.\ actually in $\mathcal{L}$. If none are, then $(M\cdot R_i) \cap \mathcal{L}$ is empty, so $R_i \cap (M^{-1}\cdot \mathcal{L})$ is empty, as desired.

    $R_1$: There are no vectors in $(M\cdot R_1) \cap \mathcal{L}$ apart from $w_1$ itself. Hence $w_1$ wins at $\frac{3\cos(\pi/7) - 1}{\sin(\pi/7)}$.

    $R_2$: The following vectors lie in $(M\cdot R_2) \cap L$, with $Mw_2$ colored in red:
    \begin{center}
    $\begin{array}{ |c|c| }
        \hline
        \text{Vector in $(M\cdot R_2) \cap L$} & \text{Slope} \\
        \hline
        (l_3,l_2) & 0.802 \\
        (l_3,2l_1) & 0.890 \\
        (2l_1,l_2) & 0.901 \\
        \textcolor{red}{(l_3,l_3)} & 1 \\
        (l_2,l_2) & 1 \\
        (l_1,l_1) & 1 \\
        (l_1+l_3,l_1+l_3) & 1 \\
        (l_1+l_2,l_1+l_2) & 1 \\
        (2l_1,2l_1) & 1 \\
        (3l_1,3l_1) & 1 \\
        \hline
    \end{array}$
    \end{center}
    The first three vectors would beat $Mw_2$ if they were holonomy vectors, since they have shallower slope. However, manually checking them at each vertex on the rectilinear surface, we find that none of them are holonomy vectors. Of the remaining vectors, all of them are either not holonomy vectors or are longer than $Mw_2$, meaning they do not beat it. Hence $(M\cdot R_2) \cap \mathcal{L}$ is empty, as desired.

    $R_3, R_4$: These follow similar arguments to $R_2$. They are each bounded regions, and we manually check vectors of $(M\cdot R_i)\cap L$ to find that $(M\cdot R_i)\cap \mathcal{L}$ is indeed empty. 
    
    $R_5$: Unlike the other four regions, $R_5$ is not bounded---it is an exceptional case to Lemma~\ref{lem:region_with_winner}, which prevents our code from working on the region. Instead, we'll show that $(M\cdot R_5) \cap \mathcal{L}$ is empty another way. From Lemma~\ref{lem:region_with_winner}, we have that $R_5 = S_\Lambda^L\left(\frac{1 + \cos(2\pi/7)}{\sin(2\pi/7)},1\right)$ is a strip of width $\csc(\pi/7)$ parallel to $(x_0,y_0)$. Then $M$ maps $R_5$ to the vertical strip $\{(x,y)\in\R^2 \mid 0 \leq x < 1, \ y>0\}$. Thus vectors in $(M\cdot R_5) \cap L$ have $x$-component in $[0,1)$. The smallest $x$-component of vectors in $L$ is $l_1 = 1$, just outside the allowed range. 
    So the only possible holonomy vectors in $M\cdot R_5$ are vertical, which means they have the same slope as $Mw_5$. 
    However, by definition, there are no holonomy vectors with the same slope as $w_5 = (x_0,y_0)$ that are shorter than it, i.e.\ that beat it. So $(M\cdot R_5) \cap \mathcal{L}$ is also empty, as desired.
\end{proof}

\begin{figure}
    \begin{tikzpicture}[scale =1.3]
    
    \coordinate (a) at (0,0);
    \coordinate (b) at (1.8019,0);
    \coordinate (c) at (1.8019,3.247);
    \coordinate (d) at (0,3.247);
    \coordinate (e) at (0,5.049);
    \coordinate (f) at (-3.247,5.049);
    \coordinate (g) at (-3.247,3.247);
    \coordinate (h) at (-2.247,3.247);
    \coordinate (i) at (-2.247,1);
    \coordinate (j) at (0,1);

    \coordinate (k) at (1.8019,1);
    \coordinate (l) at (-2.247,5.049);

    \draw[-] (a) to (b);
    \draw[-] (b) to (c);
    \draw[-] (c) to (d);
    \draw[-] (d) to (e);
    \draw[-] (e) to (f);
    \draw[-] (f) to (g);
    \draw[-] (g) to (h);
    \draw[-] (h) to (i);
    \draw[-] (i) to (j);
    \draw[-] (j) to (a);

    \draw[->] [blue, line width=1.4pt] (a) to (j);
    \draw[->] [blue, line width=1.4pt] (j) to (c);
    \draw[->] [blue, line width=1.4pt] (i) to (d);
    \draw[->] [blue, line width=1.4pt] (g) to (l);
    \draw[->] [blue, line width=1.4pt] (h) to (e);

    \node (w5) at (0.57,0.5) [text=blue] {$Mw_5$};
    \node (w4) at (-2.14,4.2) [text=blue] {$Mw_4$};
    \node (w3) at (0.32,2.2) [text=blue] {$Mw_3$};
    \node (w2) at (-1.5,2.45) [text=blue] {$Mw_2$};
    \node (w1) at (-0.7,3.9) [text=blue] {$Mw_1$};

    \end{tikzpicture}
    \caption{$M$ applied to the five winning vectors.}\label{fig:winning_vectors}
\end{figure}
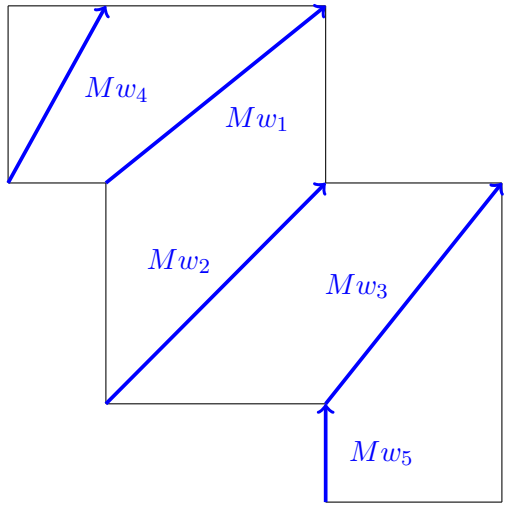

\subsection{Slope Gap Distribution}

Now that we have the complete list of winning vectors on $\Omega$, we find the winner at each point by plotting the candidacy strips for $w_1,\dots, w_5$ and taking the vector of least slope as the winner at any point where two or more strips overlap. This produces the subdivision shown in Figure~\ref{fig:subdivision}.

\begin{figure}
    \includegraphics[width=0.9\textwidth]{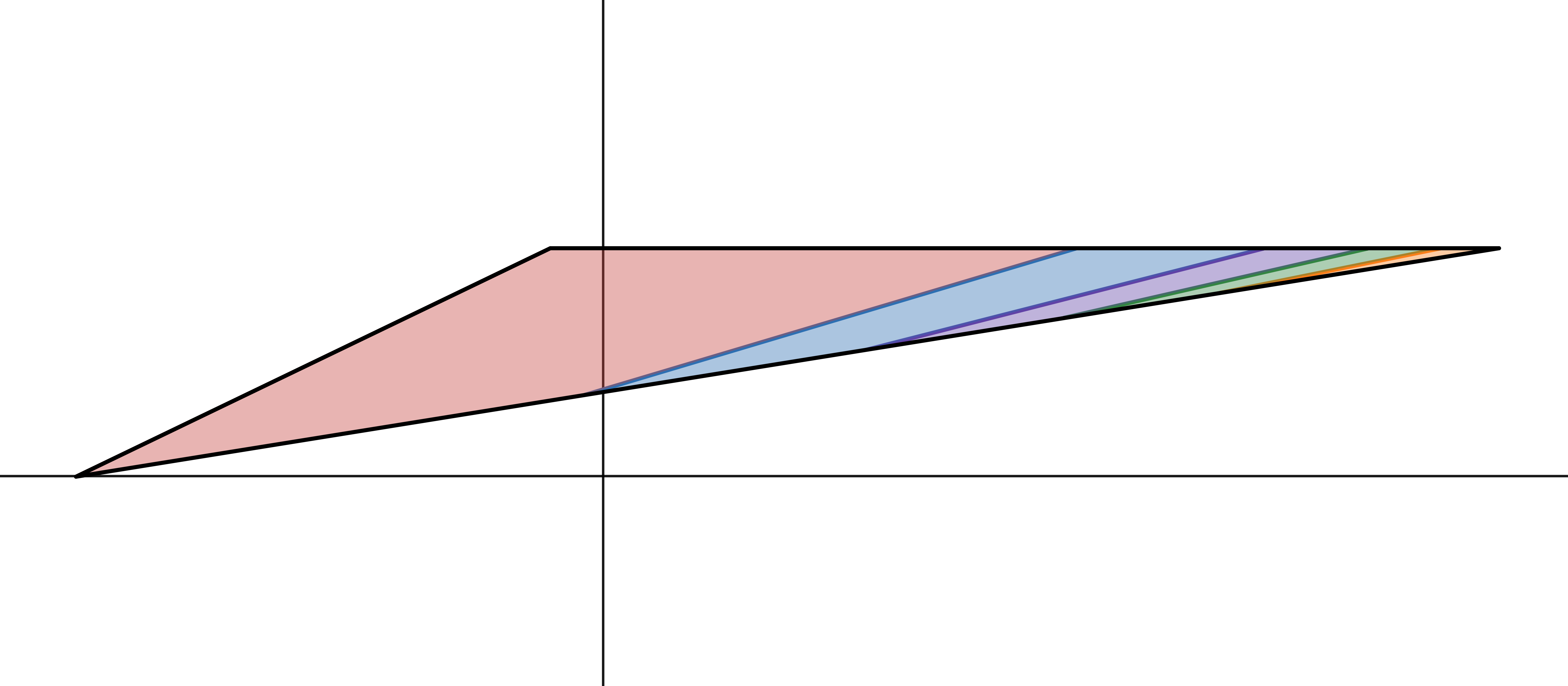}
    \caption{Subdivision of the transversal for the double heptagon into regions corresponding to different winning vectors. The vector $w_1$ wins on the orange region, $w_2$ wins on the green region, $w_3$ wins on the purple region, $w_4$ wins on the blue region, and $w_5$ wins on the red region.}
    \label{fig:subdivision}
\end{figure}

With the winning vectors and subdivision of $\Omega$ in hand, we now know the return time function at every point of the transversal. This allows us to explicitly compute the slope gap distribution for this surface as described in Section~\ref{sec:finding-pdf}. Implementing this calculation in Mathematica produces the following distribution:

\begin{thm}\label{thm:double-heptagon-sgd}
    The renormalized slope gap distribution for the double heptagon is the piecewise real analytic function\footnote{A formula for the function can be found in Appendix~\ref{sec:formula-appendix}. For the Mathematica code used to generate the graph, please contact the fourth named author.} with $13$ points of non-analyticity shown in Figure~\ref{fig:7gon_slope_gap_dist-2}.
\begin{figure}[h!]
     \centering
     \includegraphics[width=0.6\linewidth]{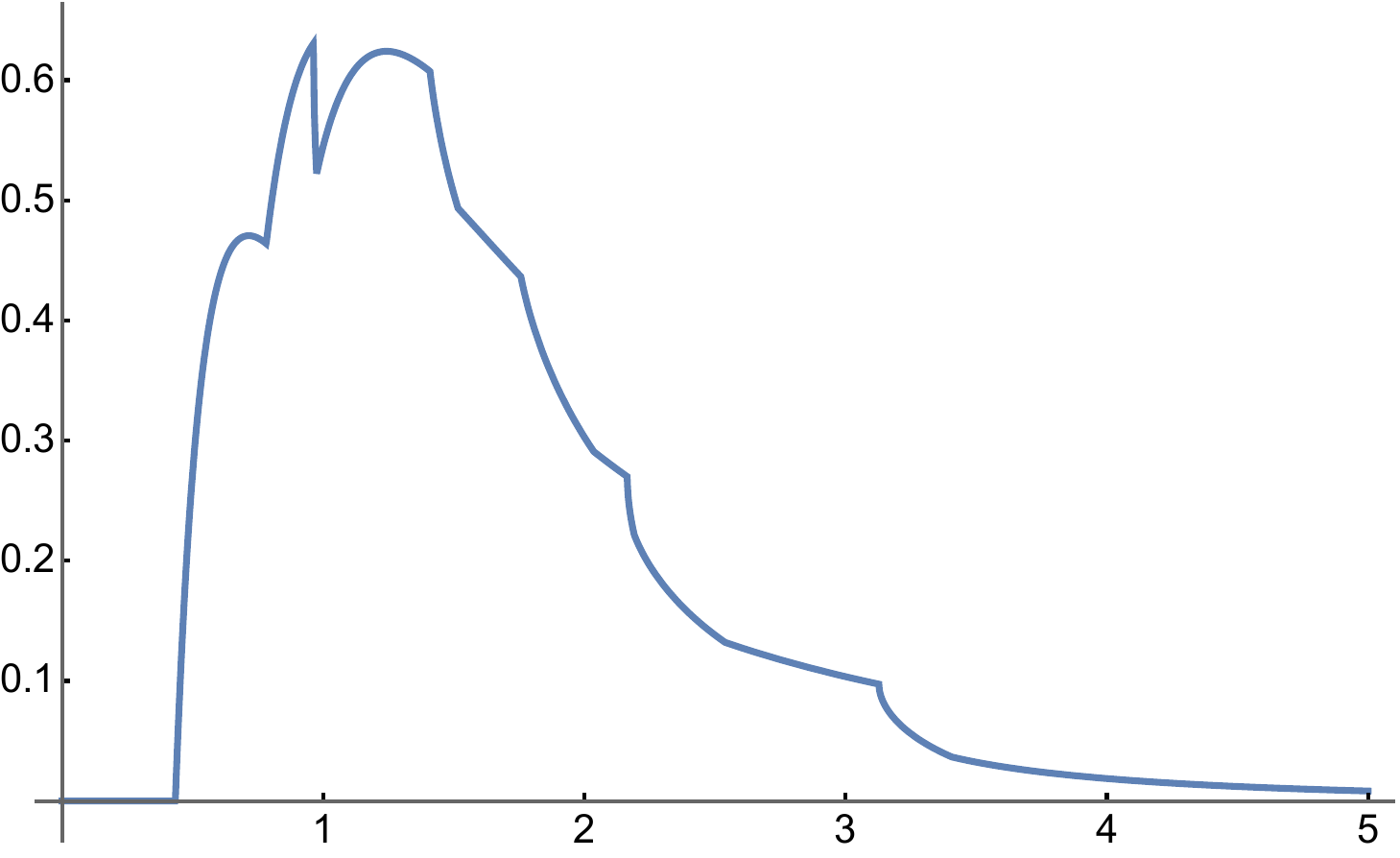}
     \caption{The slope gap distribution of the regular double heptagon translation surface.}
     \label{fig:7gon_slope_gap_dist-2}
\end{figure}
\end{thm}
The points of non-analyticity for this distribution are provided in Table~\ref{fig:non-differentiability}.

\begin{table}
\centering
\setlength{\tabcolsep}{15pt}
\renewcommand{\arraystretch}{2.5}
\begin{tabular}{|c|c|} 
\hline
 Exact return time of non-analytic point & Decimal approximation \\
 \hline
 $t_1=\sin(\frac{\pi}{7})$ & 0.433884  
 \\
 $t_2=\cos(\frac{3\pi}{14})$ & 0.781831 
 \\
 $t_3=4\sec(\frac{3\pi}{14})\sin(\frac{\pi}{7})^2$ & 0.963149 
 \\
 $t_4=\cos(\frac{\pi}{14})$  & 0.974928 
 \\
 $t_5=\cos(\frac{\pi}{14})+\sin(\frac{\pi}{7})$ & 1.40881 
 \\
 $t_6=2\cos(\frac{\pi}{14})-\sin(\frac{\pi}{7})$ & 1.51597 
 \\
 $t_7=\frac{\cos(\frac{\pi}{14})}{1-2\sin(\frac{\pi}{14})}$ & 1.75676 
 \\
 $t_8=\frac{2\cos(\frac{3\pi}{14})^3}{\sin(\frac{\pi}{14})(3-4\sin(\frac{\pi}{14}))}$ & 2.03579 
 \\
 $t_9=8\sin(\frac{\pi}{7})\sin(\frac{3\pi}{14})$ & 2.16418 
 \\
 $t_{10}=\frac{\cos(\frac{\pi}{14})\cos(\frac{3\pi}{14})}{\cos(\frac{3\pi}{14})-\sin(\frac{\pi}{7})}$ & 2.19064 
 \\
 $t_{11}=\frac{4\cos^3(\frac{\pi}{14})\cos(\frac{3\pi}{14})}{5\cos(\frac{\pi}{14})-2\cos(\frac{3\pi}{14})-5\sin(\frac{\pi}{7})}$ & 2.53859 
 \\
 $t_{12}=4\cos(\frac{3\pi}{14})$ & 3.12733
 \\
 $t_{13}=\frac{\sin(\frac{\pi}{7})}{6-8\cos(\frac{\pi}{7})+6\sin(\frac{\pi}{14})}$ & 3.40636 
 \\
 \hline
\end{tabular}
\caption{Points of non-analyticity for the slope gap distribution of the double heptagon.}
\label{fig:non-differentiability}
\end{table}

\subsection{Volume Computation}

Recall that $\slr/\slxo$ can be realized as a suspension space over $\Omega$ with roof function given by the return time function, $\mathcal{R}$ (up to a set of measure zero). Thus integrating the return time function over $\Omega$ should yield the hyperbolic volume of the space $\slr/\slxo$. In the case of the double heptagon, which has Veech group isomorphic to the $(2,7,\infty)$ triangle group, this volume is known to be $\frac{5\pi^2}{14}$. Numerically integrating our return time function over each piece of the transversal in Mathematica yields a value of $3.524858714674771$, which agrees with the theoretical value $\frac{5\pi^2}{14}$ for as many significant figures as Mathematica displays.\footnote{Code for this calculation can be found in Appendix~\ref{sec:covolume-code}.} This provides independent evidence that our subdivision of the transversal and selection of winning vectors for this surface is correct.

\subsection{Directions for Future Work}

In this paper, we have provided an algorithm which reduces the problem of finding winning vectors for a Veech surface to finding the winning vector at a finite number of points along the top edge of its Poincar\'e section (or right edge, in the parametrization of \cite{KSW}). In Lemma~\ref{lem:region_with_winner}, we also provided a general method for finding the winning vector at a point, which in many cases reduces the problem to checking a finite number of holonomy vectors which could win. 

In Section~\ref{sec:dbl_heptagon}, we used this algorithm to compute the slope gap distribution of the double heptagon translation surface, and also demonstrated a method for finding the winning vector at a point where Lemma~\ref{lem:region_with_winner} does not produce a bounded region.  
It would be interesting to investigate whether there is an algorithmic method that will always reduce the problem of finding a winning vector at a point where Lemma~\ref{lem:region_with_winner} does not produce a bounded region to a calculation with a finite number of steps. 

Note that the double heptagon is one surface in a family of surfaces, the double $n$-gons, where $n$ is odd. The ``golden-L" surface, whose slope gap distribution was computed in \cite{ACL}, is another surface in this family, as it is $\slr$-equivalent to the regular double pentagon surface. Future work might compute the slope gap distribution for this entire family of surfaces, and provide bounds on the number of points of non-analyticity, as was done for the family of $2n$-gon surfaces in \cite{Previous_SUMRY}.

\bibliography{gapsbibliography}
\bibliographystyle{acm}

\newpage
\begin{appendix}

\section{Formula for the Slope Gap Distribution of the Double Heptagon}
\label{sec:formula-appendix}

For the sake of concision, we will give separate formulas for the cumulative distribution functions coming from each region of the transversal. 

We will start by defining some frequently appearing constants:

\begin{align*}
a_1&=\cos(\frac{\pi}{14}),\\
b_1&=\sin(\frac{\pi}{14}),\\
a_2&=\cos(\frac{\pi}{7}),\\
b_2&=\sin(\frac{\pi}{7}),\\
a_3&=\cos(\frac{3\pi}{14}),\\
b_3&=\sin(\frac{3\pi}{14}),
\end{align*}
as well as some auxiliary functions:
\begin{align*}
A_1(t)&=\sqrt{1-\frac{2-2b_3}{ta_3}},\\
A_2(t)&=\sqrt{1-\frac{2a_2-2b_1}{ta_1}},\\
A_3(t)&=\sqrt{1-\frac{4b_2}{t}},\\
A_4(t)&=\sqrt{1-\frac{8b_2b_3}{t}},\\
A_5(t)&=\sqrt{1-\frac{8a_2b_2}{t}}.
\end{align*}
We will also use the labeling of times as given in the table in Figure~\ref{fig:non-differentiability}. 

The (non-normalized) cumulative distribution functions for each respective subdivision of the transversal are given by
\begin{align*}
    F_1(t)&=
    \begin{cases}
        0 & t<t_2\\
        \frac{1}{t}\left(\ln\left(\frac{a_3}{t}\right)-1\right)+\frac{1}{a_3}& t_2\leq t<t_3\\
        \frac{1}{t}\left(\ln\left(\frac{a_3}{t}\right)-1\right)+\frac{1}{a_3}+\frac{2}{t}\ath(A_1(t))-\frac{a_3A_1(t)}{2b_2^2}& t_3\leq t<t_4\\
        \frac{1}{2a_3}\left(\frac{a_3}{b_2}-2\right)^2&t\geq t_4
    \end{cases}\\\\
    F_2(t)&=
    \begin{cases}
        0 & t<t_4\\
        \frac{1}{t}\left(\ln\left(\frac{a_1}{t}\right)-1\right)+\frac{1}{a_1}& t_4\leq t<t_5\\
        \frac{1}{t}\ln\left(\frac{16b_1^2b_3^2(2a_2-2b_1-1)}{(1-4b_1b_3)^2}\right)+\frac{2}{t}\ath(A_2(t))+\frac{162b_1b_3-12b_3-15+(2b_1b_3+6b_3-4)A_2(t)}{2b_2^2}& t_5\leq t<t_6\\
        4a_1\left(3+\frac{1}{1-6b_1}\right)&t\geq t_6
    \end{cases}\\\\
    F_3(t)&=
    \begin{cases}
        0 & t<t_4\\
        \frac{1}{t}\left(\ln\left(\frac{a_1}{t}\right)-1\right)+\frac{1}{a_1}& t_4\leq t<t_7\\
        \frac{1}{t}\ln(2b_1)+\frac{2}{t}\ath(A_3(t))-\frac{A_3(t)}{2 b_2}+\frac{11-16a_2+6b_3}{4a_1b_1(1-2a_2)^2}& t_7\leq t<t_8\\
        \frac{8b_1^3}{a_1(1-2a_2)^2}&t\geq t_8
    \end{cases}\\\\
    F_4(t)&=
    \begin{cases}
        0 & t<t_2\\
        \frac{1}{t}\left(\ln\left(\frac{a_3}{t}\right)-1\right)+\frac{1}{a_3}& t_2\leq t<t_9\\
        \frac{1}{t}\left(\ln\left(\frac{a_3}{t}\right)-1\right)+\frac{1}{a_3}+\frac{4}{t}\ath(A_4(t))+\frac{b_2-a_3}{b_2^2}A_4(t)& t_9\leq t<t_{10}\\
        \frac{1}{t}\ln(1-2a_2+2b_3)+\frac{2}{t}\ath(A_4(t))-\frac{A_4(t)}{4b_2b_3}+\frac{13+10b_1-24b_3}{32b_1^2b_2^3(1+2a_2)}& t_{10}\leq t<t_{11}\\
        \frac{6a_2-2b_3-4}{b_2}&t\geq t_{11}
    \end{cases}\\\\
    F_5(t)&=
    \begin{cases}
        0 & t<t_1\\
        \frac{1}{t}\left(\ln\left(\frac{b_2}{t}\right)-1\right)+\frac{1}{b_2}& t_1\leq t<t_{12}\\
        \frac{1}{t}\left(\ln\left(\frac{b_2}{t}\right)-1\right)+\frac{1}{b_2}+\frac{4}{t}\ath(A_5(t))-\frac{4b_1b_3}{b_2}A_5(t)& t_{12}\leq t<t_{13}\\
        \frac{1}{t}\left(\ln\left(\frac{a_3}{t}\right)-1\right)+\frac{2}{t}\ath(A_5(t))-\frac{A_5(t)}{4a_2b_2}+\frac{1}{2a_3}+\frac{2b_1+1}{a_1}&t\geq t_{13}
    \end{cases}\\
\end{align*}
Then the (normalized) cumulative distribution of the double heptagon is given by
\begin{align*}
F(t) = \frac{F_1(t)+F_2(t)+F_3(t)+F_4(t)+F_5(t)}{\cot(\frac{\pi}{7})}
\end{align*}
since the total area of the transversal is $\cot(\frac{\pi}{7})$. The probability density function representing the distribution of renormalized gaps between slopes of saddle connections on the double heptagon is then
\[
f(t)=F'(t).
\]

\section{Python Code for Finding Vectors in $L$}
\label{sec:Python-L}

\begin{verbatim}
# Searches bounded regions for holonomy vectors

# Side lengths of rectilinear surface
l1 = 1
l2 = 2*cos(pi/7)
l3 = 1/(2*sin(pi/14))

# Winning vectors
w1 = (2+3*cos(2*pi/7), sin(2*pi/7))
w2 = (4*cos(pi/7)+3*cos(3*pi/7), sin(3*pi/7))
w3 = (4*cos(pi/7)+cos(3*pi/7), sin(3*pi/7))
w4 = (2+cos(2*pi/7), sin(2*pi/7))
w5 = (cos(pi/7), sin(pi/7))

# Winning vectors' images on the rectilinear surface
(u1,v1) = (l3,l2)
(u2,v2) = (l3,l3)
(u3,v3) = (l2,l3)
(u4,v4) = (l1,l2)
(u5,v5) = (0,1)

# Points on A to find winners at
a1 = 2*cos(pi/7)/sin(pi/7) + (cos(pi/7) - 1)/sin(pi/7)  
# alpha = 2 cot(pi/7)  and  (x0,y0) = w5
a2 = (w1[0] - 1)/w1[1]
a3 = (w2[0] - 1)/w2[1]
a4 = (w3[0] - 1)/w3[1]
a5 = (w4[0] - 1)/w4[1]

# Finding the box containing the bounded triangle region
def max_x(point, candidate):
  return candidate[0]/(candidate[0] - point*candidate[1])

def max_y(point, candidate):
  return candidate[1]/(candidate[0] - point*candidate[1])

# testing - confirmed, matches graph on Desmos
# max_x(a1,w1), max_y(a1,w1)

# Checking if a given pair of coords is in the image of the bounded triangle 
# region under the "Original -> Rectilinear" shear
# Can input "rect = False" to instead check the original triangle region
def in_image_region(coords, point, candidate, rect=True):
  # coordinates (s,t) of the top point of the original triangle region...
  s = max_x(point,candidate)
  t = max_y(point,candidate)
  # ...or, top point of sheared triangle region
  if rect == True:
     (s,t) = shear_to_rect(s,t)

  # check if coords are shallower than the candidate
  condition1 = s*coords[1] - t*coords[0] <= 0
  # chcek if coords are in point's candidacy strip
  condition2 = coords[1] > t/(s-1)*(coords[0]-1)

  return condition1 and condition2

# Lists a superset of the holonomy vectors in the sheared image of the 
# bounded triangle region
def image_region_vector_list(point, candidate):
  # list to hold vectors in the region
  vects = []

  count = 0

  i = 0
  j = 0
  k = 0
  l = 0
  m = 0
  n = 0

  info = ((0,0,0,0,0,0),0.0)

  # box with opposite corners at (0,0) and (xmax, ymax) contains the triangle's
  # image
  xmax = max_x(point, candidate)  
  # the shear moves the triangle in the negative x direction, but keeps it in
  # the 1st quadrant, so this works
  ymax = 1/sin(pi/7)*max_y(point, candidate)  
  # the shear stretches the y axis by a factor of csc(pi/7)

  # Main loop: goes through integer combinations of l1,l2,l3 that fit in the
  # box, filtering for ones that lie inside the triangle's image.
  # The resulting list contains all holonomy vectors that beat "candidate" at
  # "point" (if there are any).
  # It also contains vectors that are not holonomy vectors; these need to be
  # sorted out by hand.
  while i*l1 < xmax:
    while j*l2 < xmax:
      while k*l3 < xmax:

        while l*l1 < ymax:
          while m*l2 < ymax:
            while n*l3 < ymax:

              x = i*l1 + j*l2 + k*l3
              y = l*l1 + m*l2 + n*l3

              if in_image_region((x,y), point, candidate) and x != 0:

                info = ((i,j,k,l,m,n),y/x)
                vects.append(info)

                count += 1

              n+=1
            n=0
            m+=1
          m=0
          l+=1
        l=0
        k+=1
      k=0
      j+=1
    j=0
    i+=1

  # sort the list by slope
  vects.sort(key=lambda x: x[1])

  return (count, vects)

image_region_vector_list(a1, w1)

image_region_vector_list(a2, w2)

image_region_vector_list(a3, w3)

image_region_vector_list(a4, w4)
\end{verbatim}

\section{Mathematica Code for the Volume Computation}
\label{sec:covolume-code}

\begin{verbatim}

(* Constants: *)
{s1, t1} = {2 + 3 Cos[2 Pi/7], Sin[2 Pi/7]}
{s2, t2} = {4 Cos[Pi/7] + 3 Cos[3 Pi/7], Sin[3 Pi/7]}
{s3, t3} = {4 Cos[Pi/7] + Cos[3 Pi/7], Sin[3 Pi/7]}
{s4, t4} = {2 + Cos[2 Pi/7], Sin[2 Pi/7]}
{s5, t5} = {Cos[Pi/7], Sin[Pi/7]}
{x0, y0} = {s5, t5}
Alpha = 2 Cot[Pi/7]

(* Edge equations, corner coordinates, and return time function for \
each region: *)

(* Region 1: *)

(* top edge: *) y == 1
(* right edge: *) x == (x0/y0 + Alpha) y - 1/y0
(* left edge: *)  x == (s1 y - 1)/t1
(* return time function: *) R1[x_, y_] := t1 /(y (s1 y - t1 x))
(* corners: *)
Bottom1 := Solve[
  (* left edge: *)  x == (s1 y - 1)/t1
   && 
   (* right edge: *) x == (x0/y0 + Alpha) y - 1/y0
  ,
  {x, y}]
Topright1 := Solve[
  (* top edge: *) y == 1
   && 
    (* right edge: *) x == (x0/y0 + Alpha) y - 1/y0
  ,
  {x, y}]
Topleft1 := Solve[
  (* top edge: *) y == 1
   && 
   (* left edge: *)  x == (s1 y - 1)/t1
  , 
  {x, y}]


(* Region 2: *)

(* top edge: *) y == 1
(* right edge: *) x == (s1  y - 1)/t1
(* left edge: *) x == (s2 y - 1)/t2
(* bottom edge: *) x == (x0/y0 + Alpha) y - 1/y0
(* return time function: *) R2[x_, y_] := t2 /(y (s2 y - t2 x))
(* corners: *)
Bottom2 := Solve[
  (* left edge: *) x == (s2 y - 1)/t2
   && 
   (* bottom edge: *) x == (x0/y0 + Alpha) y - 1/y0
  ,
  {x, y}]
Middle2 := Solve[
  (* right edge: *) x == (s1  y - 1)/t1
   && 
   (* bottom edge: *) x == (x0/y0 + Alpha) y - 1/y0
  ,
  {x, y}]
Topright2 := Solve[
  (* top edge: *) y == 1
   && 
    (* right edge: *) x == (s1  y - 1)/t1
  ,
  {x, y}]
Topleft2 := Solve[
  (* top edge: *) y == 1
   && 
   (* left edge: *) x == (s2 y - 1)/t2
  , 
  {x, y}]
Extracoord2 := Solve[
  (* same y as middle corner: *) y == (y /. First[Middle2])
   && 
    (* left edge: *) x == (s2 y - 1)/t2
  ,
  {x, y}]


(* Region 3: *)

(* top edge: *) y == 1
(* right edge: *) x == (s2  y - 1)/t2
(* left edge: *) x == (s3 y - 1)/t3
(* bottom edge: *) x == (x0/y0 + Alpha) y - 1/y0
(* return time function: *) R3[x_, y_] := t3 /(y (s3 y - t3 x))
(* corners: *)
Bottom3 := Solve[
  (* left edge: *) x == (s3 y - 1)/t3
   && 
   (* bottom edge: *) x == (x0/y0 + Alpha) y - 1/y0
  ,
  {x, y}]
Middle3 := Solve[
  (* right edge: *) x == (s2  y - 1)/t2
   && 
   (* bottom edge: *) x == (x0/y0 + Alpha) y - 1/y0
  ,
  {x, y}]
Topright3 := Solve[
  (* top edge: *) y == 1
   && 
    (* right edge: *) x == (s2  y - 1)/t2
  ,
  {x, y}]
Topleft3 := Solve[
  (* top edge: *) y == 1
   && 
   (* left edge: *) x == (s3 y - 1)/t3
  , 
  {x, y}]
Extracoord3 := Solve[
  (* same y as middle corner: *) y == (y /. First[Middle3])
   && 
    (* left edge: *) x == (s3 y - 1)/t3
  ,
  {x, y}]


(* Region 4: *)

(* top edge: *) y == 1
(* right edge: *) x == (s3  y - 1)/t3
(* left edge: *) x == (s4 y - 1)/t4
(* bottom edge: *) x == (x0/y0 + Alpha) y - 1/y0
(* return time function: *) R4[x_, y_] := t4 /(y (s4 y - t4 x))
(* corners: *)
Bottom4 := Solve[
  (* left edge: *) x == (s4 y - 1)/t4
   && 
   (* bottom edge: *) x == (x0/y0 + Alpha) y - 1/y0
  ,
  {x, y}]
Middle4 := Solve[
  (* right edge: *) x == (s3  y - 1)/t3
   && 
   (* bottom edge: *) x == (x0/y0 + Alpha) y - 1/y0
  ,
  {x, y}]
Topright4 := Solve[
  (* top edge: *) y == 1
   && 
    (* right edge: *) x == (s3  y - 1)/t3
  ,
  {x, y}]
Topleft4 := Solve[
  (* top edge: *) y == 1
   && 
   (* left edge: *) x == (s4 y - 1)/t4
  , 
  {x, y}]
Extracoord4 := Solve[
  (* same y as middle corner: *) y == (y /. First[Middle4])
   && 
    (* left edge: *) x == (s4 y - 1)/t4
  ,
  {x, y}]


(* Region 5: *)

(* top edge: *) y == 1
(* right edge: *) x == (s4  y - 1)/t4
(* left edge: *) x == (s5 y - 1)/t5
(* bottom edge: *) x == (x0/y0 + Alpha) y - 1/y0
(* return time function: *) R5[x_, y_] := t5/(y (s5 y - t5 x))
(* corners: *)
Bottom5 := Solve[
  (* left edge: *) x == (s5 y - 1)/t5
   && 
   (* bottom edge: *) x == (x0/y0 + Alpha) y - 1/y0
  ,
  {x, y}]
Middle5 := Solve[
  (* right edge: *) x == (s4  y - 1)/t4
   && 
   (* bottom edge: *) x == (x0/y0 + Alpha) y - 1/y0
  ,
  {x, y}]
Topright5 := Solve[
  (* top edge: *) y == 1
   && 
    (* right edge: *) x == (s4  y - 1)/t4
  ,
  {x, y}]
Topleft5 := Solve[
  (* top edge: *) y == 1
   && 
   (* left edge: *) x == (s5 y - 1)/t5
  , 
  {x, y}]
Extracoord5 := Solve[
  (* same y as middle corner: *) y == (y /. First[Middle5])
   && 
    (* left edge: *) x == (s5 y - 1)/t5
  ,
  {x, y}]

(* Integrate each region's return time function over the region, 
then sum those results *)

(* REGION 1 *)

R1volume = 
 Integrate[
  Integrate[
   R1[x, y], (* region 1 return time function *)
   {x, (s1  y - 1)/t1, (x0/y0 + Alpha) y - 1/y0}],(* x: 
  from left edge to right edge*)
  {y, y /. First[Bottom1], 1}] (* y: from bottom corner to top edge *)

(* REGION 2 *)

R2volume = 
 (* trapezoid *)
 Integrate[
   Integrate[
    R2[x, y],
    {x, (s2 y - 1)/t2, (s1 y - 1)/t1}],(* x: 
   from left edge to right edge *)
   {y, y /. First[Middle2], 1} (* y: from middle corner to top edge *)
   ] +
  (* triangle *)
  Integrate[
   Integrate[
    R2[x, y],
    {x, (s2 y - 1)/t2, (x0/y0 + Alpha) y - 1/y0}],(* x: 
   from left edge to bottom edge *)
   {y, y /. First[Bottom2], y /. First[Middle2]}]  (* y: 
  from bottom corner to middle corner *)

(* REGION 3 *)

R3volume = 
 (* trapezoid *)
 Integrate[
   Integrate[
    R3[x, y],
    {x, (s3 y - 1)/t3, (s2 y - 1)/t2}],(* x: 
   from left edge to right edge *)
   {y, y /. First[Middle3], 1}   (* y: 
   from middle corner to top edge *)
   ] +
  (* triangle *)
  Integrate[
   Integrate[
    R3[x, y],
    {x, (s3 y - 1)/t3, (x0/y0 + Alpha) y - 1/y0}],(* x: 
   from left edge to bottom edge *)
   {y, y /. First[Bottom3], 
    y /. First[Middle3]}]  (* y: from bottom corner to middle corner *)

(* REGION 4 *)

R4volume = 
 (* trapezoid *)
 Integrate[
   Integrate[
    R4[x, y],
    {x, (s4 y - 1)/t4, (s3 y - 1)/t3}],(* x: 
   from left edge to right edge *)
   {y, y /. First[Middle4], 1} (* y: from middle corner to top edge *)
   ] +
  (* triangle *)
  Integrate[
   Integrate[
    R4[x, y],
    {x, (s4 y - 1)/t4, (x0/y0 + Alpha) y - 1/y0}],(* x: 
   from left edge to bottom edge *)
   {y, y /. First[Bottom4], 
    y /. First[Middle4]}]  (* y: from bottom corner to middle corner *)

(* REGION 5 *)

R5volume = 
 (* trapezoid *)
 Integrate[
   Integrate[
    R5[x, y],
    {x, (s5 y - 1)/t5, (s4 y - 1)/t4}],(* x: 
   from left edge to right edge *)
   {y, y /. First[Middle5], 1}  (* y: from middle corner to top edge *)
   ] +
  (* triangle *)
  Integrate[
   Integrate[
    R5[x, y],
    {x, (s5 y - 1)/t5, (x0/y0 + Alpha) y - 1/y0}],(* x: 
   from left edge to bottom edge *)
   {y, y /. First[Bottom5], y /. First[Middle5]}]  (* y: 
  from bottom corner to middle corner *)

N[R1volume + R2volume + R3volume + R4volume + R5volume]

N[R1volume + R2volume + R3volume + R4volume + R5volume - 5 Pi^2/14]

Simplify[R1volume + R2volume + R3volume + R4volume + R5volume]
\end{verbatim}

\vspace{1em}

\end{appendix}

\end{document}